\documentclass{lmcs}
\pdfoutput=1

\usepackage{lastpage}
\lmcsdoi{20}{1}{12}
\lmcsheading{}{\pageref{LastPage}}{}{}%
{Nov.~09,~2022}{Feb.~07,~2024}{}

\usepackage[utf8]{inputenc}
\usepackage{amssymb}
\usepackage{hyperref}
\usepackage{enumitem}
\newcommand{\m}{\mathbf} 

\newcommand{\bimp}{\mathbin{-\!\!*}}

\newcommand{\jn}{\vee}
\newcommand{\mt}{\wedge}
\newcommand{\ncdot}{\cdot}
\newcommand{\icdot}{\cdot}
\newcommand{\lcdot}{\!\ncdot} 
\newcommand{\rcdot}{\ncdot\!} 
\newcommand{\nstar}{*}
\newcommand{\p}{\mathsf{p}}
\newcommand{\q}{\mathsf{q}}

\newcommand{\da}{\mathord{\downarrow}}

\usepackage{tikz}
\usepackage{forest} 
\tikzstyle{every node} = [draw, fill=white, ellipse, inner sep=0pt, minimum size=15pt]
\tikzstyle{n} = [draw=none, rectangle, inner sep=0pt] 
\tikzstyle{i} = [draw, fill=black, circle, inner sep=0pt, minimum size=15pt]
\tikzstyle{d} = [thick,dotted]

\title[Unary-determined distributive \texorpdfstring{$\ell$}{l}-magmas]{Varieties of unary-determined distributive \texorpdfstring{\\$\ell$}{l}-magmas and bunched implication algebras}
\author[N.~Alpay]{Natanael Alpay\lmcsorcid{0000-0002-1505-7531}}[a]
\author[P.~Jipsen]{Peter Jipsen\lmcsorcid{0000-0001-8608-808X}}[b]
\author[M.~Sugimoto]{Melissa Sugimoto\lmcsorcid{0009-0004-7168-8954}}[b]
\address{University of California Irvine, California 92697, USA} 
\email{nalpay@uci.edu}
\address{Chapman University, Orange, California 92866, USA}
\email{jipsen@chapman.edu, msugimoto@chapman.edu}

\keywords{distributive lattice-ordered magmas, bunched implication algebras, idempotent semirings, enumerating finite models}

\begin{document}

\begin{abstract}
A distributive lattice-ordered magma ($d\ell$-magma) $(A,\wedge,\vee,\cdot)$ is a distributive lattice with a binary operation $\cdot$ that preserves joins in both arguments, and when
$\cdot$ is associative then $(A,\vee,\cdot)$ is an idempotent semiring. A $d\ell$-magma with a top $\top$ is \emph{unary-determined} if $x\ncdot  y=(x\rcdot \top\wedge y)$ $\vee(x\wedge \top\lcdot y)$. These algebras are term-equivalent to a subvariety of distributive lattices with $\top$ and two join-preserving unary operations $\p,\q $. We obtain simple conditions on $\p,\q $
such that $x\ncdot  y=(\p x\wedge y)\vee(x\wedge \q y)$ is associative, commutative, idempotent and/or has an identity element.

This generalizes previous results on the structure of doubly idempotent semirings and, in
the case when the distributive lattice is a Heyting algebra, it provides structural insight into unary-determined algebraic models of bunched implication logic.
We also provide Kripke semantics for the algebras under consideration, which leads to more efficient algorithms for constructing finite models. We find all subdirectly irreducible algebras up to cardinality eight in which $\p=\q $ is a closure operator, as well as all finite unary-determined bunched implication chains and map out the poset of join-irreducible varieties generated by them.
\end{abstract}

\maketitle

\section{Introduction}
Idempotent semirings $(A,\vee,\cdot)$ play an important role in several areas of computer science, such as network optimization, formal languages, Kleene algebras and program semantics. In this setting they are often assumed to have constants $0,1$ that are the additive and multiplicative identity respectively, with $0$ also being an absorbing element.
However semirings are usually only assumed to have two binary operations $+,\cdot$ that
are associative such that $+$ is also commutative and $\cdot$ distributes over $+$
from the left and right \cite{HW1998}.
A semiring is (additively) idempotent if $x+x=x$, hence
$+$ is a (join) semilattice, and \emph{doubly idempotent} if $x\cdot x=x$ as well. If
$\cdot$ is also commutative, then it defines a meet semilattice. The special case when these two semilattices coincide
corresponds exactly to the variety of distributive lattices, which have a well understood
structure theory.

In \cite{AJ2020} a complete structural description was given for finite commutative doubly idempotent semirings where either the multiplicative semilattice is a chain, or the additive semilattice is a Boolean algebra. Here we show that the second description can
be significantly generalized to the setting where the additive semilattice is a distributive
lattice, dropping the assumptions of finiteness, multiplicative commutativity and idempotence
in favor of the algebraic condition $x\ncdot  y=(\p x\wedge y)\vee(x\wedge \q y)$ for two unary join-preserving operations $\p,\q $. While this property is quite restrictive in general, it does hold in all idempotent Boolean magmas and expresses a binary operation in terms of two simpler unary operations. A full structural description of all (finite) idempotent semirings is unlikely, but in the setting of unary-determined idempotent semirings progress is possible.

In Section 2 we provide the needed background and prove a term-equivalence between
a subvariety of top-bounded $d\ell$-magmas and a subvariety of top-boun\-ded distributive lattices with two unary operators. This is then specialized to cases where $\cdot$ is
associative, commutative, idempotent or has an identity element.
In the next section we show that when the distributive lattice is a Brouwerian algebra
or Heyting algebra, then $\cdot$ is residuated if and only if both $\p$ and $\q $ are
residuated. This establishes a connection with bunched implication algebras (BI-algebras)
that are the algebraic semantics of bunched implication logic \cite{OP1998}, used in the setting of separation logic for program verification, including reasoning about pointers \cite{Rey2002} and concurrent processes \cite{OHe2007}. Section 4 contains Kripke semantics
for $d\ell$-magmas, called Birkhoff frames, and for the two unary operators $\p,\q $. This establishes the connection to the previous results in \cite{AJ2020} and leads to 
the main result (Thm.~\ref{assoc}) that preorder forest $P$-frames capture a larger
class of multiplicatively idempotent BI-algebras and doubly idempotent semirings. Although the heap models of BI-algebras used in applications are not (multiplicatively) idempotent, they contain idempotent subalgebras and homomorphic images, hence a characterization of unary-determined idempotent BI-algebras does provide insight into the general case.
In the next section we define weakly conservative $\ell$-magmas and their corresponding frames.
In Section 6 we apply the results from the previous sections to count the number of preorder forest $P$-frames up to isomorphism if their partial order is an antichain and also if it is a chain. Finally in Section 7 we calculate all subdirectly irreducible algebras up to cardinality eight in which $\p$ and $\q $ are the same closure operator, and map out the poset of join-irreducible varieties generated by them.

\section{A term-equivalence between distributive lattices with operators}

A \emph{distributive lattice-ordered magma}, or \emph{$d\ell$-magma}, is an algebra $\m A=(A,\wedge,\vee,\cdot)$ such that
$(A,\wedge,\vee)$ is a distributive lattice and $\cdot$ is a \emph{binary operator}, which in this case means a binary operation that distributes over $\vee$, i.e., $x\cdot (y\vee z)=x\cdot y\vee x\cdot z$ and $(x\vee y)\cdot z = x\cdot z\vee y\cdot z$ for all $x,y,z\in A$. Throughout it is assumed that $\cdot$ binds more strongly than $\wedge,\vee$, and as usual the lattice order $\le$ is defined by $x\le y\iff x\wedge y=x$ $(\iff x\vee y=y)$. If the distributive lattice has a top element $\top$ or a bottom element $\bot$ then it is called \emph{$\top$-bounded} or
\emph{$\bot$-bounded}, or simply \emph{bounded} if both exist.
A $d\ell$-magma $\m A$ is \emph{normal} and $\cdot$ is a \emph{normal} operation if $A$ is $\bot$-bounded and satisfies $x\ncdot \bot=\bot=\bot\ncdot x$. 
Similarly, a unary operation $f$ on $A$ is an \emph{operator} if it satisfies the identity $f(x\vee y)=fx\vee fy$, 
and it is \emph{normal} if $f\bot=\bot$. 
For brevity and to reduce the number of nested parentheses, we write function application as $fx$ rather than $f(x)$, with the convention that it has priority over $\cdot$ hence, e.g., $fx\cdot y=(f(x))\ncdot y$ (this convention ensures unique readability).
Note that since operators distribute over $\vee$ in each argument, they are order-preserving in each argument. The operation $f$
is said to be \emph{inflationary} if $x\le fx$ for all $x\in A$.

A binary operation $\cdot$ is said to be \emph{idempotent}
if $x\ncdot x=x$ for all $x\in A$, \emph{commutative} if $x\ncdot y=y\ncdot x$ and \emph{associative} if $(x\ncdot y)\ncdot z=x\ncdot (y\ncdot z)$. A \emph{semigroup} is a set with an associative operation, a \emph{monoid} is a semigroup with an identity element denoted by $1$, a \emph{band} is
a semigroup that is also idempotent, and a \emph{semilattice} is a commutative band.
As usual, a semilattice is partially ordered by $x\sqsubseteq y\iff x\icdot y=x$, and in this
case $x\ncdot y$ is the meet operation with respect to $\sqsubseteq$. We also use this
terminology with the prefix $d\ell$, in which case the magma operation satisfies the
corresponding identities.

A $d\ell$-magma is called \emph{unary-determined} if it is $\top$-bounded and satisfies the identity
\[x\ncdot y=(x\rcdot \top\wedge y)\vee(x\wedge\top\lcdot y).\]
As examples, we mention that all doubly-idempotent semirings with a Boolean join-semilattice are unary-determined (see Lemma~\ref{BAud}). Complete and atomic versions of such semirings are studied in \cite{AJ2020}, and the results from that paper are generalized here to unary-determined $d\ell$-magmas with algebraic proofs that apply to all members of the variety, while the previous results applied only to complete and atomic algebras.

A \emph{$d\ell\p\q$-algebra} is a $\top$-bounded distributive lattice with
two unary operators $\p,\q$ that satisfy
\[
x\wedge \p\top\le \q x, \quad x\wedge \q\top\le \p x.
\]
We note that throughout $\p ,\q $ denote unary operations, and they bind more strongly than $\cdot,\wedge,\vee$.
These two (in)equational axioms are needed for
our first result which shows that unary-determined $d\ell$-magmas and $d\ell \p \q$-algebras
are term-equivalent. This means that although the two varieties are based on different
sets of fundamental operations (called the signature of each class), each fundamental operation of an algebra in one variety is identical to a term-operation constructed from fundamental operations of an algebra in the other variety (and vice versa). From the point of view of category theory, term-equivalent varieties are model categories of the same Lawvere theory.

Note that the (in)equalities above are satisfied in any $\top$-bounded distributive lattice with inflationary operators $\p ,\q $ since then $\p \top=\top=\q  \top$. A \emph{$d\ell \p $-algebra} is a $d\ell \p \q $-algebra that satisfies the identity $\p x=\q x$.

Although unary-determined $d\ell$-magmas and $d\ell \p \q $-algebras seem rather special, they are simpler than general $d\ell$-magmas, yet include interesting idempotent semirings (as reducts).

\begin{thm}\label{pqtermeq} \hfill
\begin{enumerate}[label=\textup{(\arabic*)}]
\item Let $(A,\wedge,\vee,\top,\p ,\q )$ be a $d\ell \p \q $-algebra
and define $x\ncdot y=(\p x\wedge y)\vee(x\wedge \q y)$. Then $(A,\wedge,\vee,\top,\cdot)$ is a unary-determined $d\ell$-magma and
$\p ,\q $ are given by $\p x=x\rcdot \top$ and $\q x=\top\ncdot  x$.
\item Let $(A,\wedge,\vee,\top,\cdot)$ be a unary-determined $d\ell$-magma and define $\p x=x\rcdot \top$, $\q x=\top\ncdot  x$. Then $(A,\wedge,\vee,\top,\p ,\q )$ is a $d\ell \p \q $-algebra and $\cdot$ is definable from
$\p ,\q $ via $x\ncdot  y=(\p x\wedge y)\vee(x\wedge \q y)$.
\end{enumerate}
\end{thm}
\begin{proof}
(1) Assume $\p ,\q $ are unary operators on a $\top$-bounded distributive lattice $(A,\wedge,\vee,\top)$, and $x\icdot y=(\p x\wedge y)\vee(x\wedge \q y)$. Then
\begin{align*}
x\icdot (y\vee z)
&=(\p x\wedge (y\vee z))\vee(x\wedge \q (y\vee z))\\
&=(\p x\wedge y)\vee (\p x\wedge z)\vee(x\wedge \q y)\vee(x\wedge \q z)\\
&=(\p x\wedge y)\vee(x\wedge \q y)\vee (\p x\wedge z)\vee(x\wedge \q z)\\
&=x\icdot y\vee x\icdot z.
\end{align*}
A similar calculation shows that $(x\vee y)\ncdot z=x\ncdot z\vee y\ncdot z$, hence $\cdot$ is an operator.

Since $\p ,\q $  satisfy $x\wedge \q \top\le \p x$, it follows that $x\rcdot \top=(\p x\wedge\top)\vee(x\wedge \q \top)=\p x\vee(x\wedge \q \top)=\p x$, and similarly
$\top\lcdot x = \q x$ is implied by $x\wedge \p \top\le \q x$. Now the identity $x\ncdot y=(x\rcdot \top\wedge y)\vee(x\wedge \top\lcdot y)$ holds by definition.

(2) Assume $(A,\wedge,\vee,\top,\cdot)$ is a unary-determined $d\ell$-magma, and define $\p x=x\rcdot \top$, $\q x=\top\lcdot y$.
Then $\p ,\q $ are unary operators and $\p x=x\rcdot \top=(x\rcdot \top\wedge \top)\vee(x\wedge \top\ncdot \top)=\p x\vee(x\wedge \q \top)$, hence
$x\wedge \q \top\le \p x$. The inequality $x\wedge \p\top\le \q x$ is proved similarly.
The operation $\cdot$ can be recovered from $\p,\q $ since $x\ncdot y=(\p x\wedge y)\vee(x\wedge \q y)$ follows from the identity we assumed.
\end{proof}

The preceding theorem shows that unary-determined $d\ell$-magmas and $d\ell \p\q$-algebras are ``essentially the same'', and we can choose to work with the signature that is preferred in a given situation. The unary operators of $d\ell \p\q$-algebras are simpler to handle, while the binary operator $\cdot$ is familiar in the semiring setting. Next we examine how standard properties of $\cdot$ are captured by identities in the language of $d\ell \p\q$-algebras.

\begin{lem}\label{pqprops}
Let $(A,\wedge,\vee,\top,\p,\q )$ be a $d\ell \p\q$-algebra
and define $x\ncdot  y=(\p x\wedge y)\vee(x\wedge \q y)$.
\begin{enumerate}[label=\textup{(\arabic*)}]
\item The operator $\cdot$ is commutative if and only if $\p=\q $.
\item If $\p=\q $, then $\cdot$ is associative if and only if $\p((\p x\wedge y)\vee (x\wedge \p y))=(\p x\wedge \p y)\vee (x\wedge \p\p y)$.
\item The operator $\cdot$ is idempotent if and only if $\p$ and $\q $ are inflationary, if and only if $\p\top=\top=\q \top$.
\item If $\cdot$ is idempotent, then it is associative if and only if
\begin{align*}
\p((\p x\wedge y)\vee (x\wedge \q y))&=(\p x\wedge \p y)\vee (x\wedge \q y)\text{ and}\\
\q ((\p x\wedge y)\vee (x\wedge \q y))&=(\p x\wedge y)\vee (\q x\wedge \q y).
\end{align*}
\item The operator $\cdot$ has an identity $1$ if and only if $\p1{=}\top{=}\q 1$ and $(\p x\vee \q x)\wedge 1\le x$.
\item If $\cdot$ has an identity, then $\cdot$ is idempotent.
\end{enumerate}
\end{lem}
\begin{proof}
(1)
Assuming $x\ncdot y=y\ncdot x$, we clearly have $x\rcdot \top=\top\lcdot x$, hence $\p x=\q x$. The converse makes use
of commutativity of $\wedge$ and $\vee$:
$x\icdot y=(\p x\wedge y)\vee(x\wedge \p y)=(\p y\wedge x)\vee(y\wedge \p x)=y\icdot x$.

(2)
Assume $\p=\q $. If $\cdot$ is associative then $(x\icdot y)\icdot \top=x\icdot (y\icdot \top)$, so by the previous theorem,
$\p(x\icdot y)=x\icdot \p y$, which translates to
\[
\p((\p x\wedge y)\vee (x\wedge \p y))=(\p x\wedge \p y)\vee (x\wedge \p\p y)\quad (*).
\]

Conversely, suppose $(*)$ holds, and note that $\p(x\icdot y)=\p(y\icdot x)$ by (1), hence
\[
\p((\p x\wedge y)\vee (x\wedge \p y))=(\p x\wedge \p y)\vee (\p\p x\wedge y)=(\p x\wedge \p y)\vee (x\wedge \p\p y)\vee(\p\p x\wedge y)\quad (**).
\]

It suffices to prove $(x\icdot y)\icdot z\le x\icdot (y\icdot z)$ since then $z\icdot (y\icdot x)\le (z\icdot y)\icdot x$ follows by commutativity. Now
\begin{align*}
(x\icdot y)\icdot z
&=[\p((\p x\wedge y)\vee(x\wedge \p y))\wedge z]\vee[((\p x\wedge y)\vee(x\wedge \p y))\wedge \p z]\\
&=[((\p x\wedge \p y)\vee (x\wedge \p\p y))\wedge z]\vee[\p x\wedge y\wedge \p z]\vee[x\wedge \p y\wedge \p z]\text{ using $(*)$}\\
&=[\p x\wedge \p y\wedge z]\vee [x\wedge \p\p y\wedge z]\vee[\p x\wedge y\wedge \p z]\vee[x\wedge \p y\wedge \p z]\\
&\le[\p x\wedge \p y\wedge z]\vee[\p x\wedge y\wedge \p z]\vee[x\wedge \p y\wedge \p z]\vee[x\wedge y\wedge \p\p z]\vee[x{\wedge}\p\p y{\wedge}z]\\
&=[\p x\wedge \p y\wedge z]\vee[\p x\wedge y\wedge \p z]\vee[x\wedge((\p y\wedge \p z)\vee(y\wedge \p\p z)\vee(\p\p y\wedge z))]\\
&=[\p x\wedge((\p y\wedge z)\vee(y\wedge \p z))]\vee[x\wedge \p((\p y\wedge z)\vee(y\wedge \p z))]\text{ using $(**)$}\\
&=x\icdot (y\icdot z).
\end{align*}

(3) If $\cdot$ is idempotent, then $x=x\icdot x\le x\rcdot \top=\p x$ and $x\le \top\lcdot x=\q x$. Conversely,
if $\p,\q $ are inflationary then $x\icdot x=(\p x\wedge x)\vee(x\wedge \q y)=x\vee x=x$,
hence $\cdot$ is idempotent. For the second equivalence, if $\p\top=\top=\q \top$, then $\p,\q $ are inflationary since they satisfy $x\wedge \p\top\le \q x$ and $x\wedge \q \top\le \p x$. The reverse implication holds because $x\le \p x,\q x$ implies $\top\le \p\top,\q \top$.

(4) Assume $\cdot$ is idempotent and associative. Then $(\top\lcdot x)\icdot \top=\top\icdot (x\rcdot \top)$, hence
$\q \p x=\p\q x$. Furthermore, $\p\q x=\top\lcdot x\rcdot \top=\top\icdot x\icdot x\icdot \top=(\q x)\icdot (\p x)=(\p\q x\wedge \p x)\vee (\q x\wedge \q \p x)$. By (3) $\p,\q $ are inflationary, so $\p x\le \p\q x$ and $\q x\le \q \p x$. Therefore
$\p\q x=\p x\vee \q x$. Now we translate $(x\icdot y)\icdot \top=x\icdot (y\icdot \top)$ to obtain $\p(x\icdot y)=x\icdot (\p y)$, hence
\begin{align*}
&\p((\p x\wedge y)\vee (x\wedge \q y))
=(\p x\wedge \p y)\vee (x\wedge \q \p y)
=(\p x\wedge \p y)\vee (x\wedge (\p y\vee \q y))\\
&=(\p x\wedge \p y)\vee(x\wedge \p y)\vee(x\wedge \q y)
=(\p x\wedge \p y)\vee(x\wedge \q y)\text{ since $x\le \p x$ by (3).}
\end{align*}
The identity $\q ((\p x\wedge y)\vee (x\wedge \q y))=(\p x\wedge y)\vee (\q x\wedge \q y)$ has a similar proof.

Conversely, assume the two identities hold. Then using distributivity
\begin{align*}
(x\icdot y)\icdot z
&=[\p((\p x\wedge y)\vee(x\wedge \q y))\wedge z]\vee[((\p x\wedge y)\vee(x\wedge \q y))\wedge \q z]\\
&=[\p x\wedge \p y\wedge z]\vee[x\wedge \q y\wedge z]\vee[\p x\wedge y\wedge \q z]\vee[x\wedge \q y\wedge \q z]\\
&=[\p x\wedge \p y\wedge z]\vee[\p x\wedge y\wedge \q z]\vee[x\wedge \q y\wedge \q z]\text{ \ since $x\wedge \q y\wedge z\le x\wedge \q y\wedge \q z$}\\
&=[\p x\wedge \p y\wedge z]\vee[\p x\wedge y\wedge \q z]\vee[x\wedge \p y\wedge z]\vee[x\wedge \q y\wedge \q z]\\
&=[\p x\wedge((\p y\wedge z)\vee(y\wedge \q z))]\vee[x\wedge \q ((\p y\wedge z)\vee(y\wedge \q z))]=x\icdot (y\icdot z).
\end{align*}

(5) Assume $x$ has an identity $1$. Then $\p1=1\icdot \top=\top=\top\icdot  1=\q 1$ and $x=x\icdot 1=(\p x\wedge 1)\vee(x\wedge \q 1)=(\p x\wedge 1)\vee x$, so $\p x\wedge 1\le x$ and similarly $\q x\wedge 1\le x$. Therefore $(\p x\vee \q x)\wedge 1=(\p x\wedge 1)\vee(\q x\wedge 1)\le x$.

Conversely, suppose $\p1=\top=\q 1$ and $(\p x\vee \q x)\wedge 1\le x$. Then $x\icdot 1=(\p x\wedge 1)\vee(x\wedge \q 1)= (\p x\wedge 1)\vee x=x$ since $\p x\wedge 1\le x$. Likewise $1\icdot x=x$.

(6) This follows from (3) since $x=x\icdot 1\le x\rcdot \top=\p x$ and $x=1\icdot x\le \q x$.
\end{proof}

Note that if $\m A$ also has a bottom bound $\bot$, then $\p,\q $ are normal if and only if $\cdot$ is normal, hence the term-equivalence preserves normality.

This term-equivalence is useful since distributive lattices with unary operators are considerably simpler than distributive lattices
with binary operators. In particular, (2) and (4) show that associativity can be replaced by one or two 2-variable identities in this variety. This provides more efficient ways to construct associative operators from a (pair of) unary operator(s) on a distributive lattice.
The variety of $\top$-bounded distributive lattices is obtained as a subvariety of $d\ell \p\q $-algebras that satisfy $\p x=x=\q x$, or a subvariety of unary determined $d\ell$-magmas that
satisfy $x\cdot y=x\wedge y$.

For small cardinalities, Table~\ref{nofmagmas} shows the number of algebras that are unary-determined (shown in the even numbered rows) for several subvarieties of normal $d\ell$-magmas. As seen from rows 7-10, under the assumption of associativity, commutativity and idempotence of $\cdot$, the property of being unary-determined is a relatively mild restriction
compared to the general case of normal $d\ell$-magmas.

\begin{table}
\begin{center}
\begin{tabular}{|c|l|ccccccc|}\hline
&\hfill Cardinality $n=$ \ & 2& 3& 4& 5& 6& 7& 8\\\hline
1&normal $d\ell$-magmas& 2& 20& 1116 & & & &\\
2&normal $d\ell \p\q $-algebras& 2& 6& 46& 3435& & &\\\hline
3&normal comm. $d\ell$-magmas& 2& 10& 148& 3554&&&\\
4&normal $d\ell \p$-algebras& 2& 4& 15& 46& 183& 688& \\\hline
5&normal comm. $d\ell$-semigroups& 2& 8& 57& 392& 3212 &&\\
6&normal assoc. $d\ell \p$-algebras& 2& 4& 13& 35& 109& 315& 998 \\\hline
7&normal comm. idem. $d\ell$-semigroups&1& 2& 8& 25& 97& 366&\\
8&normal assoc. idem. $d\ell \p$-algebras& 1& 2& 7& 18& 57& 163& 521\\\hline
9&normal comm. idem. $d\ell$-monoids &1& 2& 6& 15& 44& 115& 326\\
10&normal assoc. idem. $d\ell \p1$-algebras& 1& 2& 5& 10& 24& 47& 108 \\\hline
11&distributive lattices& 1& 1& 2& 3& 5& 8& 15\\
\hline
\end{tabular}
\end{center}
\caption{The number of algebras of cardinality $n$ up to isomorphism.}\label{nofmagmas}
\end{table}
A \emph{Boolean magma} is a Boolean algebra with a binary operator. The next
lemma shows that if the operator is idempotent, then it is always unary-deter\-mined, hence the results in the current paper generalize the theorems about idempotent Boolean nonassociative quantales in \cite{AJ2020}.

\begin{lem}\label{BAud}
Every idempotent Boolean magma $(A,\wedge,\vee,\neg,\bot,\!\top,\cdot)$
is unary-de\-termined, i.e., satisfies $x\ncdot y=(x\rcdot \top\wedge y)\vee(x\wedge\top\lcdot y)$.
\end{lem}
\begin{proof}
Idempotence is equivalent to $x\wedge y\le x\ncdot y\le x\vee y$ since
$(x\wedge y)\ncdot (x\wedge y)\le x\ncdot y\le (x\vee y)\ncdot (x\vee y)$ holds in all partially ordered algebras
where $\cdot$ is an order-preserving binary operation. The following calculation
\begin{align*}
x\rcdot \top\wedge y&=x\ncdot (y\vee\neg y)\wedge y=(x\ncdot y\wedge y)\vee (x\ncdot \neg y\wedge y)\\
&\le x\ncdot y\vee ((x\vee\neg y)\wedge y)=x\ncdot y\vee (x\wedge y)\vee(\neg y\wedge y)=x\ncdot y
\end{align*}
and a similar one for $x\wedge\top\lcdot y\le x\ncdot y$ prove that $x\ncdot y\ge (x\rcdot \top\wedge y)\vee(x\wedge\top\lcdot y)$.

Using Boolean negation, the opposite inequality is equivalent to
\[x\ncdot y\wedge\neg(x\rcdot \top\wedge y)\le x\wedge \top\lcdot y.\]
By De Morgan's law it suffices to show $(x\ncdot y\wedge\neg(x\rcdot \top))\vee(x\ncdot y\wedge \neg y)\le x\wedge \top\lcdot y$. Since $x\ncdot y\le x\rcdot \top$, the first meet disappears. Next, by idempotence, $x\ncdot y\wedge\neg y\le (x\vee y)\wedge\neg y=(x\wedge\neg y)\vee (y\wedge\neg y)\le x$ and finally $x\ncdot y\wedge\neg y\le x\ncdot y\le\top\lcdot y$.
\end{proof}

\section{BI-algebras from Heyting algebras and residuated unary operations}

We now recall some basic definitions about residuated operations, adjoints and residuated lattices. For an overview and additional details we refer to \cite{GJKO2007}.
A \emph{Brouwerian algebra} $(A,\wedge,\vee,\to,\top)$ is a $\top$-bounded lattice
such that $\to$ is the \emph{residual} of $\wedge$, i.e.,
\[
x\wedge y\le z\quad\iff\quad y\le x\to z.
\]
Since $\to$ is the residual of $\wedge$, we have that $\wedge$ is join-preserving, so
the lattice is distributive \cite[Lem. 4.1]{GJKO2007}. The $\top$-bound is included as a constant since it always exists
when a meet-operation has a residual: $x\wedge y\le x$ always holds, hence $y\le (x\to x)=\top$.
A \emph{Heyting algebra} is a bounded Brouwerian algebra with a constant $\bot$ denoting the bottom element.

A \emph{dual operator} is an $n$-ary operation on a lattice that preserves meets in each argument.
A \emph{residual} or \emph{upper adjoint} of a unary operation $\p$ on a poset $\m A=(A,\le)$ is a unary operation $\p^*$ such that
\[
\p x\le y\iff x\le \p^*y
\]
for all $x,y\in A$. If $\m A$ is a lattice, then the existence of a residual guarantees that $\p$ is an operator and $\p^*$ is a dual operator \cite[Lem. 3.5]{GJKO2007}. Moreover, if $\m A$ is bounded, then $\p\bot=\bot$ and
$\p^*\top=\top$.

A binary operation $\cdot$ on a poset is \emph{residuated} if there exist a \emph{left residual} $\backslash$ and a \emph{right residual} $/$ such that
\[
x\cdot y\le z\iff y\le x\backslash z\iff x\le z/y.
\]
A \emph{residuated $\ell$-magma} $(A,\wedge,\vee,\cdot,\backslash,/)$ is a lattice with a residuated binary operation. In this case $\cdot$ is an operator and $\backslash, /$ are dual
operators in the ``numerator'' argument. The ``denominator'' arguments of $\backslash, /$ map joins to meets, hence they are order reversing. A \emph{residuated Brouwerian-magma} is a residuated $\ell$-magma expanded with $\to,\top$ such that $(A,\wedge,\vee,\to,\top)$ is a Brouwerian algebra.

A \emph{residuated lattice} is a residuated $\ell$-magma with $\cdot$ associative and a constant $1$ that is an identity element, i.e., $(A,\cdot,1)$ is a monoid. A \emph{generalized bunched implication algebra}, or GBI-algebra, $\m A=(A,\wedge,\vee,\to,\top,\cdot,1,\backslash,/)$  is
a $\top$-bounded residuated lattice with a residual $\to$ for the meet operation, i.e., $(A,\wedge,\vee,\to,\top)$ is a Brouwerian algebra. A GBI-algebra is called a \emph{bunched implication algebra} (BI-algebra) if $\cdot$ is commutative and $\m A$ also has a bottom element, denoted by the constant $\bot$, hence a BI-algebra has a Heyting algebra reduct. These algebras are the algebraic semantics for
bunched implication logic, which is the propositional part of separation logic, a Hoare logic used for reasoning about memory references in computer programs. In this setting the operation $\cdot$ is usually denoted by $*$, the left residual $\backslash$ is denoted $\bimp$, and $/$ can be omitted since $x/y=y\bimp x$.

Note that the property of being a residual can be expressed by inequalities ($\p^*$ is a residual of $\p$ if and only if $\p(\p^*x)\le x\le \p^*(\p x)$ for all $x$, and $\p,\p^*$ are order preserving), hence the classes of all Brouwerian algebras, Heyting algebras, residuated $\ell$-magmas, residuated Brouwerian-magmas, residuated lattices, (G)BI-algebras, and pairs of residuated unary maps on a lattice are varieties (see e.g.~\cite[Thm.~2.7 and Lem.~3.2]{GJKO2007}). Recall also that a $\top$-bounded magma is unary-determined if it satisfies the identity $x\icdot y=(x\rcdot \top\wedge y)\vee(x\wedge \top\lcdot y)$.

We are now ready to prove a result that upgrades the term-equivalence of Theorem~\ref{pqtermeq} to Brouwerian algebras with two pairs of residuated maps and unary-determined residuated Brouwerian-magmas.

\begin{thm} \hfill
\begin{enumerate}[label=\textup{(\arabic*)}]
\item Let $(A,\wedge,\vee,\to,\top,\p,\p^*,\q ,\q ^*)$ be a Brouwerian algebra with unary operators $\p,\q $ and their residuals $\p^*,\q ^*$ such that $x\wedge \p\top\le \q x$, $x\wedge \q \top\le \p x$. If we define
$x\ncdot y=(\p x\wedge y)\vee(x\wedge \q y)$,
\[
x\backslash y=(\p x\to y)\wedge \q ^*(x\to y)\quad\text{and}\quad
x/y=\p^*(y\to x)\wedge (\q y\to x),
\]
then $(A,\wedge,\vee,\top,\cdot,\backslash,/)$ is a unary-determined residuated Brouwerian-magma  
and the unary operations are recovered by $\p x=x\rcdot \top$, $\p^*x=x/\top$, $\q x=\top\lcdot x$ and $\q ^*x=\top\backslash x$.

\item Let $(A,\wedge,\vee,\to,\top,\cdot,\backslash,/)$ be a unary-determined
residuated Brouwerian-mag\-ma and define $\p x=x\rcdot \top$, $\p^*x=x/\top$,
$\q x=\top\lcdot x$ and $\q ^*x=\top\backslash x$. Then $(A,\wedge,\vee,\to,$ 
$\top,\p,\p^*,\q ,\q ^*)$ is a Brouwerian algebra with a unary operators
$\p,\q $ and dual operators $\p^*,\q ^*$ that satisfies $x\wedge \p\top\le \q x$, 
$x\wedge \q \top\le \p x$.
\end{enumerate}
\end{thm}
 
\begin{proof} (1) The following calculation shows that $\cdot$ is residuated.
\begin{alignat*}{2}
x\cdot y\le z
&\iff (\p x\wedge y)\vee(x\wedge \q y)\le z
&&\iff \p x\wedge y\le z\text{ and }x\wedge \q y\le z\\
&\iff y\le \p x\to z\text{ and } y\le \q ^*(x\to z)
&&\iff y\le (\p x\to z)\wedge \q ^*(x\to z)
\end{alignat*}
hence $x\backslash z=(\p x\to z)\wedge \q ^*(x\to z)$ and similarly $z/y=\p^*(y\to z)\wedge (\q y\to z)$. By Theorem~\ref{pqtermeq} it follows that $\p x=x\rcdot \top, \q x=\top\lcdot x$ and $x\icdot y=(x\rcdot \top\wedge y)\vee(x\wedge\top\lcdot y)$. Since $x\rcdot \top \le y\iff x\le y/\top$ we obtain $\p^*(x)=x/\top$, and similarly $\q ^*(x)=\top\backslash x$.

(2) Since $\cdot$ is residuated it follows that $\p^*$ and $\q ^*$ are the unary residuals of $\p$, $\q $ respectively. The remaining parts hold by Theorem~\ref{pqtermeq}.
\end{proof}

Recall that a closure operator $\p$ is an order-preserving unary function on a poset such that $x\le \p x=\p\p x$. A bounded $d\ell \p$-algebra where $\p$ is a normal closure operator is called a $d\ell \p$-closure algebra.
If $\cdot$ is idempotent and associative then $x\rcdot \top=x\ncdot(\top\ncdot\top)=(x\ncdot\top)\ncdot\top$, so $\p x=x\rcdot \top$ is a closure operator.

\begin{lem}\label{five} Assume $\m A$ is a $d\ell \p$-closure algebra and let $x\ncdot  y=(\p x\wedge y)\vee(x\wedge \p y)$.
Then $\cdot$ is associative if and only if
$\p x\wedge \p y\le \p((\p x\wedge y)\vee(x\vee \p y))$.
\end{lem}
\begin{proof}
By Lemma~\ref{pqprops} $\cdot$ is associative if and only if
the identity $\p((\p x\wedge y)\vee(x\vee \p y))=(\p x\wedge \p y)\vee(x\wedge \p y)$ holds.
This is equivalent to $\p x\wedge \p y\le \p((\p x\wedge y)\vee(x\vee \p y))$
since $x\wedge \p y\le \p x\wedge \p y$, $\p(\p x\wedge y)\le \p\p x\wedge \p y=\p x\wedge \p y$
and similarly $\p(x\wedge \p y)\le \p x\wedge \p y$.
\end{proof}
We note that there exist non-associative $d\ell\p$-algebras, as shown (later) by the algebra $\m D_{12}$ in Figure~\ref{fig:sidlp}.
The preceding theorems specialize to a term-equivalence for a subvariety of unary-determined BI-algebras as follows:
\begin{cor}\label{BItermeq} \hfill
\begin{enumerate}
\item Let $(A,\wedge,\vee,\to,\top,\bot,\p,\p^*,1)$ be a Heyting algebra with a closure operator $\p$, residual $\p^*$ and constant $1$ such that $\p x\wedge \p y\le \p((\p x\wedge y)\vee(x\wedge \p y))$, $\p1=\top$ and $\p x\wedge 1\le x$.
If we define
$x\nstar y=(\p x\wedge y)\vee(x\wedge \p y)$ and $x\bimp y=(\p x\to y)\wedge \p^*(x\wedge y)$
then $(A,\wedge,\vee,\top,\to,*,\bimp,1)$ is a unary-determined BI-algebra and $(x\nstar \!\top)\wedge (y\nstar \!\top)\le(((x\nstar \!\top)\wedge y)\vee(x\wedge (y\nstar \!\top)))\nstar \!\top$ holds.

\item Let $(A,\wedge,\vee,\to,\top,\bot,*,\bimp,1)$ be a unary-determined BI-algebra, 
and define $\p x=x\nstar \!\top$ and $\p^*x=\top\bimp x$.
Then $(A,\wedge,\vee,\to,$ $\top,\bot,\p,\p^*,1)$ is a Heyting algebra with a closure operator
$\p$ that has $\p^*$ as residual and satisfies $\p x\wedge \p y\le \p((\p x\wedge y)\vee(x\wedge \p y))$,
$\p1=\top$ and $\p x\wedge 1\le x$.
\end{enumerate}
\end{cor}

Heyting algebras with a closure operator provide algebraic semantics for \textbf{IntS4}$_\Diamond$ \cite{Dos1985}, an intuitionistic modal logic with an $S4$-modality. Hence the result above establishes a connection between certain extensions of bunched implication logic and of intuitionistic modal logic.

By Lemma~\ref{pqprops}(6) unary-determined BI-algebras satisfy $x\nstar x=x$, which does not hold in BI-algebras that model applications (e.g., heap storage). However, as mentioned in the introduction, they are members of the variety of BI-algebras, and understanding their properties via this term-equivalence is useful for the general theory. E.g., structural results about algebraic objects (such as rings) often start by investigating the idempotent algebras, followed by sets of idempotent elements in more general algebras. 
Line~10 in Table~\ref{nofmagmas} also shows that finite unary-determined BI-algebras are not rare (algebras with normal join-preserving operators can be uniquely expanded with residuals in the finite case, hence expansions of the algebras counted in Line~10 are indeed term-equivalent to unary-determined BI-algebras).

\section{Relational semantics for \texorpdfstring{$d\ell$}{dl}-magmas}

We now briefly recall relational semantics for bounded distributive lattices with
operators and then apply correspondence theory to derive first-order
conditions for the equational properties of the preceding sections.

An element in a lattice is \emph{completely join-irreducible} if it is not the supremum
of all the elements strictly below it. The set of all completely join-irreducible elements of a lattice $A$ is denoted by $J(A)$, and it is partially ordered by restricting the order of $A$ to $J(A)$.
For example, if $A$ is a Boolean lattice, then $J(A)=At(A)$ is the antichain of \emph{atoms}, i.e., all elements immediately above the bottom element. The set $M(A)$ of completely meet-irreducible elements is defined dually. A lattice is \emph{perfect} if it is complete (i.e., all joins and meets exist) and every element is a join of completely join-irreducibles and a meet of completely meet-irreducibles. For a Boolean algebra, the notion of perfect is equivalent to being \emph{complete} (i.e., joins and meets of all subsets exist) and \emph{atomic} (i.e., every non-bottom element has an atom below it).

Recall that for a poset $\m W=(W,\le)$, a \emph{downset} is a subset $X$ such that $y\le x\in X$ implies $y\in X$. As in modal logic, $W$ is considered a set of ``worlds'' or states. We let $D(\m W)$ be the set of all downsets of $\m W$, and $(D(\m W),\cap,\cup)$ the \emph{lattice of downsets}. The collection $D(\m W)$
is a perfect distributive lattice with infinitary meet and join given by (arbitrary) intersections and unions. The following result, due to Birkhoff \cite[Thm. III.3.3]{Bir1967} for lattices of finite height, shows that up to isomorphism all perfect distributive lattices arise in this way. The poset $J(D(\m W))$ contains exactly the principal downsets $\da x=\{y\in W\mid y\le x\}$.

\begin{thmC}[{\cite[10.29]{DP2002}}] For a lattice $A$ the following are equivalent:
\begin{enumerate}
\item $A$ is distributive and  perfect.
\item $A$ is isomorphic to the lattice of downsets of a partial order.
\end{enumerate}
\end{thmC}

Note that the set of upsets of a poset is also a perfect distributive lattice,
and if it is ordered by reverse inclusion then this lattice is isomorphic to 
the downset lattice described above. 
It is also well known that the maps $J$ and $D$
are functors for a categorial duality between the category of posets with order-preserving maps and the category of perfect distributive lattices
with complete lattice homomorphisms (i.e., maps that preserve arbitrary joins and meets).

A \emph{complete operator} on a complete lattice is an operation that, in each argument, is completely join-preserving, while a \emph{complete dual operator} is completely meet-preserving (in each argument). A lattice-ordered algebra is called \emph{perfect} if its lattice reduct is perfect and every fundamental operation on it is a complete operator or dual operator.
The duality between the category of perfect distributive lattices and posets extends to the category of perfect distributive lattices with (a fixed signature of) complete operators and dual operators. The corresponding poset category has additional relations of arity $n+1$ for each (dual) operator of arity $n$, and the relations have to be upward or downward closed in each argument. For example, a binary relation $Q\subseteq W^2$ is upward closed in the second argument if $xQy\le z\implies xQz$. Here $xQy\le z$ is an abbreviation for $xQy$ and $y\le z$.

Perfect distributive lattices with operators, their residuals and dual operators are algebraic models for many
logics, including relevance logic, intuitionistic logic, H\'ajek's basic logic, \L ukasiewicz
logic and bunched implication logic \cite{GNV2005,GJKO2007}.
In such an algebra $\m A$, a join-preserving binary operation is determined by
a ternary relation $R$ on $J(\m A)$ given by
\[
xRyz\iff x\le yz.
\]
The notation $xRyz$ is shorthand for $(x,y,z)\in R$.
For $b,c\in A$ the product $bc$ is recovered as $\bigvee\{x\in J(\m A)\mid xRyz$ for some $y\le b$ and $z\le c\}$.

The relational structure $(J(\m A),\le,R)$ is an example of a Birkhoff frame. In general, a \emph{Birkhoff frame} \cite{GJ2020} is a triple $\m W=(W,\le,R)$ where $(W,\le)$ is a poset, and $R\subseteq W^3$ satisfies the following three properties (downward closure in the 1st, and upward closure in the 2nd and 3rd argument):
\begin{center}
\begin{tabular}{l}
	$\text{(R1)}\ u\le xRyz\implies uRyz$\quad\\
	$\text{(R2)}\ xRyz~\&~y\leq v\implies xRvz$\quad\\
	$\text{(R3)}\ xRyz~\&~z\leq w\implies xRyw$.
\end{tabular}
\end{center}
A Birkhoff frame $\m W$ defines the downset algebra $\m D(\m W)=(D(\m W),\bigcap,\bigcup,\cdot)$ by
\[
Y\cdot Z=\{x\in W\mid xRyz\text{ for some $y\in Y$ and $z\in Z$}\}.
\]
The  property (R1) ensures that $Y\cdot Z\in D(\m W)$.

In relevance logic \cite{DR2002} similar ternary frames are known as Routley-Meyer frames. In that setting upsets are used to recover the distributive lattice-ordered relevance algebra, and this choice implies that $J(A)$ with the induced order from $A$ is dually isomorphic to $(W,\le)$. Another difference is that Routley-Meyer frames have a unary relation and axioms to ensure it is a left identity element of the $\cdot$ operation.

The duality between perfect $d\ell$-magmas and Birkhoff frames is recalled below. Here we assume that
the binary operation on a complete $d\ell$-magma is a complete operator, i.e., distributes over arbitrary joins in each argument. Such algebras are also known as \emph{nonassociative quantales} or \emph{prequantales} \cite{Ros1990}.

\begin{thmC}[\cite{GJ2020}]
\begin{enumerate}
\item If $\m A$ is a perfect $d\ell$-magma and $R\subseteq J(A)^3$ is defined by $xRyz\Leftrightarrow x\le yz$
 then $J(\m A)=(J(A),\le,R)$ is a Birkhoff frame, and $\m A\cong \m D(J(\m A))$.

\item If $\m W$ is a Birkhoff frame then $\m D(\m W)$ is a perfect $d\ell$-magma, and $\m W\cong (J(D(\m W)),\subseteq,R_{\da})$, where $(\da x,\da y,\da z)\in R_{\da}\Leftrightarrow xRyz$.
\end{enumerate}
\end{thmC}

A ternary relation $R$ is called \emph{commutative} if $xRyz\implies xRzy$ for all $x,y,z$. The justification for this terminology is provided by the following result.

\begin{lem}
For any Birkhoff frame $\m{W}$, $\m D(\m{W})$ is commutative if and only if $R$ is commutative.
\end{lem}

\begin{lem}\label{idem}
	Let $\m{W}$ be a Birkhoff frame. Then $\m D(\m{W})$ is idempotent if and only if
	$xRxx$ and $(xRyz \implies x\leq y \text{ or }x\leq z)$ for all $x,y,z\in W$.
\end{lem}

\begin{proof}
Assume $\m D(\m{W})$ is idempotent, and let $x\in W$. Then $\da x \cdot \da x = \da x$ since $\da x \in D(\m W)$. From $x\in \da x$ we deduce $x\in \da x \cdot \da x$, whence it
follows that $xRyz$ for some $y\in \da x ,z\in \da x$.	Therefore $xRyz$ for $y\leq x, z\leq x$, which implies $xRxx$ by (R2) and (R3).

Next assume $xRyz$ holds. Then $x\in\da\{y,z\}\cdot \da\{y,z\}=\da\{y,z\}$ by idempotence.
Hence for some $w\in \{y,z\}$ we have $x\leq w$, and it follows that $x\leq y$ or $x\leq z$.

For the converse, assume $xRxx$ and $(xRyz\implies x\le y$ or $x\le z)$ for all
$x,y,z\in W$ and let $X\in D(\m W)$. 
From $xRxx$ we obtain $X\subseteq X\cdot X$.

For the reverse inclusion, let $x\in X\cdot X$. Then $xRyz$ holds for some $y,z\in X$.
By assumption $xRyz$ implies $x\leq y$ or $x\leq z$. Since $X$ is a downset,
$x\leq y\implies x\in X$ and $x\leq z\implies x\in X$. Hence $X\cdot X=X$.
\end{proof}

The previous two results are examples of correspondence theory, since they show that
an equational property on a perfect $d\ell$-magma corresponds to
a first-order condition on its Birkhoff frame.

The relational semantics of a perfect $d\ell \p\q $-magma is given by a \emph{PQ-frame}, which
is a partially-ordered relational structure $(W,\le,P,Q)$ such that $P,Q$ are binary relations
on $W$, $u\le xPy\le v\implies uPv$ and $u\le xQy\le v\implies uQv$.
Relations with this property are called \emph{weakening relations} \cite{KV2016,GJ2020}, and this
is what ensures that if we define 
\[\p(Y)=\{x\mid \exists y(xPy\ \&\ y\in Y)\}\]
for a downset
$Y$, then $\p$ is a complete normal join-preserving operator that produces a downset,
and $P$ is uniquely determined by $xPy \Leftrightarrow x\in \p(\da y)$.
Similarly, a normal operator $\q $ is defined from $Q$, and uniquely determines $Q$.
The residual $\p^*$ of $\p$ is a completely meet-preserving operator, defined
by $\p^*(Y)=\{x\mid \forall y(yPx\Rightarrow y\in Y)\}$, and likewise for $\q ^*$. If $P=Q$ then
we omit $Q$ and refer to $(W,\le, P)$ simply as a \emph{$P$-frame}.

We now list some correspondence results for $d\ell \p\q $-magmas. We begin
with a theorem that restates the term-equivalence of Theorem~\ref{pqtermeq}
as a definitional equivalence on frames. A direct proof of this result is straightforward,
but it also follows from Theorem~\ref{pqtermeq} by correspondence theory.

\begin{thm}\label{equivframes} \hfill
\begin{enumerate}[label=\textup{(\arabic*)}]
\item Let $(W,\le,P,Q)$ be a $PQ$-frame such that $x\le y\ \&\ xPz\Rightarrow xQy$ and
$x\le y\ \&\ xQz\Rightarrow xPy$. If we define
$
xRyz\Leftrightarrow (xPy\ \&\ x\le z)\text{ or }(x\le y\ \&\ xQz)
$
then $(W,\le, R)$ is a Birkhoff frame, and $P,Q$ are obtained from $R$ via
$xPy\Leftrightarrow \exists w(xRyw)$ and $xQy\Leftrightarrow\exists w(xRwy)$.

\item Let $(W,\le,R)$ be a Birkhoff frame that satisfies
$
xRyz\Leftrightarrow (\exists w(xRyw)\ \&\ x\le z)\text{ or }(x\le y\ \&\ \exists w(xRwz))
$
and define $xPy\Leftrightarrow \exists w(xRyw)$, \ $xQy\Leftrightarrow\exists w(xRwy)$. Then $(W,\le,P,Q)$ is a $PQ$-frame in which $x\le y\ \&\ xPz\Rightarrow xQy$ and $x\le y\ \&\ xQz\Rightarrow xPy$ hold.
\end{enumerate}
\end{thm}
Note that the universal formula $x\le y \ \&\ xPz \implies xQy$ corresponds
to the $d\ell \p\q $-magma axiom  $Y\wedge \p\top\le \q Y$.

A significant advantage of $PQ$-frames over Birkhoff frames is that binary relations have a
graphical representation in the form of directed graphs (whereas ternary relations
are 3-ary hypergraphs that are more complicated to draw).
Equational properties from Lemma~\ref{pqprops}, Cor.~\ref{BItermeq} correspond
to the following first-order properties on $PQ$-frames. 

\begin{lem}\label{frameprops}
Assume $\m A$ is a perfect $d\ell \p\q $-algebra and $\m W=(W,\le, P, Q)$ is its
corresponding $PQ$-frame. The constant $1\in A$ (when present) is assumed to correspond to a downset $E\subseteq W$. Then
\begin{enumerate}[label=\textup{(\arabic*)}]
\item $a\le \p a$ holds in $\m A$ if and only if $P$ is reflexive,
\item $\p\p a\le \p a$ holds in $\m A$ if and only if $P$ is transitive,
\item $\p a=\q a$ holds in $\m A$ if and only if $P=Q$,
\item $\p 1=\top$ holds in $\m A$ if and only if $\forall x\exists y(y\in E\ \&\ xPy)$ holds in $\m W$,
\item $\p a\wedge 1\le a$ holds in $\m A$ if and only if $x\in E\ \&\ xPy\Rightarrow x\le y$
holds in $\m W$,
\item $\p a\wedge \p b\le \p((\p a\wedge b)\vee(a\wedge \p b))$ holds in $\m A$ if and only if
\[wPx\ \&\ wPy\Rightarrow \exists v(wPv\ \&\ (vPx\ \&\ v\le y\text{ or }v\le x\ \&\ vPy))\quad\text{holds in $\m W$.}\]
\end{enumerate}
\end{lem}

\begin{proof} (1)--(3) These correspondences are well known from modal logic.

(4) For $x\in W$ and $E=\da 1$ we have $x\le \p 1$ if and only if there exists $y\in W$
such that $y\le 1$ and $x\le \p y$, or equivalently, $y\in E$ and $xPy$.

(5) In the forward direction, let $a = \da y$. Then it follows that $x \in \p(\da y) \cap E$ implies $x\in\da y$, and consequently $x\in E \ \&\ xPy\implies x \leq y$.

In the reverse direction, let $Y$ be a downset of $W$ and assume $x \in \p Y \cap E$. Then $x \in E$ and $xPy$ for some $y \in Y$. Hence $x \leq y$, or equivalently $x \in \da y \subseteq Y$. Thus, $\p Y \cap E \subseteq Y$, so the algebra $\m A$ satisfies $\p a\wedge 1\le a$ for all $a\in A$.

(6) In the forward direction, let $a=\da x$ and $b=\da y$. Then it follows from the inequality that $w\in \p\da x\cap \da y\implies w\in \p((\p\da x\cap \da y)\cup(\da x\cap \p\da y))$ for all $w\in W$. This in turn implies
$wPx\ \&\ wPy\implies \exists v(wPv\ \&\ v\in (\p\da x\cap \da y)\cup(\da x\cap \p\da y))$, which translates to the given first-order condition.

In the reverse direction, let $X,Y$ be downsets of $W$ and assume $w\in \p X\cap \p Y$. Then $wPx$ and $wPy$ for some $x\in X$ and $y\in Y$. It follows that there exists a $v\in W$ such that
$(wPv\ \&\ (vPx\ \&\ v\le y\text{ or }v\le x\ \&\ vPy))$, hence $v\in (\p X\cap Y)\cup(X\cap \p Y)$. Therefore $w\in \p(\p X\cap Y)\cup(X\cap \p Y)$.
\end{proof}

Recall that a ternary relation $R$ is commutative if $xRyz\Leftrightarrow xRzy$ for all $x,y$. From Theorem~\ref{equivframes} we also obtain the following result.

\begin{cor}
Let $(W,\leq, P,Q)$ be a $PQ$-frame and define $R$ as in Thm.~\ref{equivframes}(1). Then $R$ is commutative if and only if $P=Q$.
\end{cor}

This corollary shows that in the commutative setting a $PQ$-frame only needs one of the two
binary relations. Hence we define $\m W=(W,\le,P)$ to be a \emph{P-frame} if $P$ is a weakening relation, i.e., $u\le xPy\le v\implies uPv$. 

We now turn to the problem of ensuring that the binary operation of a $d\ell$-magma is associative.
For Birkhoff frames the following characterization of associativity is well known from relation algebras \cite{Ma1982}
(in the Boolean case) and from the Routley-Meyer semantics for relevance logic \cite{DR2002} in general.

\begin{lem}
	Let $\m W=(W,\leq, R)$ be a Birkhoff frame. Then $\m D(\m W)$ is an associative $\ell$-magma if and only if $\forall wxyz(\exists u (uRxy\,\&\,wRuz)\Leftrightarrow  \exists v(vRyz\,\&\,wRxv))$.
	If $R$ is commutative, then the equivalence can be replaced by the implication
$\forall uwxyz(uRxy\,\&\,wRuz\Rightarrow  \exists v(vRyz\,\&\,wRxv))$.
\end{lem}
This lemma is another correspondence result that follows from translating $w\in (XY)Z\Leftrightarrow w\in X(YZ)$ for $X,Y,Z\in D(\m W)$. In the commutative case $(XY)Z\subseteq X(YZ)$ implies the reverse inclusion, hence only one of the implications is needed.
We now show that for a large class of $P$-frames the 6-variable universal-existential formula for associativity can be replaced by simpler universal formulas with only three variables.

A \emph{preorder forest $P$-frame} is a $P$-frame such that $P$ is a preorder (i.e., reflexive and transitive) and satisfies the formula
\begin{center}
$\qquad xPy \text{ and } xPz \implies  x\le y\text{ or }x\le z\text{ or }yPz\text{ or }zPy$.\qquad (Pforest)
\end{center}

Note that since $P$ is a weakening relation, reflexivity of $P$ implies that ${\le}\subseteq P$ because $xPx$ and $x\le y$ implies $xPy$.

It is interesting to visualize the properties that define preorder forest $P$-frames by implications between
Hasse diagrams with $\le$-edges (solid) and $P$-edges (dotted) as in Figure~\ref{P-frame}.
However, one needs to keep in mind that dotted lines could be horizontal (if $xPy$ and $yPx$) and
that any line could be a loop if two variables refer to the same element.

\begin{figure}
\begin{center}
\tikzstyle{every picture} = [scale=.25, baseline=0pt]
\tikzstyle{every node} = [draw=none, rectangle, inner sep=1pt] 
\tikzstyle{t} = [thick]
\tikzstyle{d} = [thick, densely dotted]
\quad(Pforest)
\begin{tikzpicture}
\node(0)at(1,-1){$x$};
\node(1)at(0,1){$y$};
\node(2)at(2,1){$z$};
\draw[d](0)--(1);
\draw[d](0)--(2);
\end{tikzpicture}
$\implies$
\begin{tikzpicture}
\node(0)at(1,-1){$x$};
\node(1)at(0,1){$y$};
\node(2)at(2,1){$z$};
\draw[t](0)--(1);
\draw[d](0)--(2);
\end{tikzpicture}
 \ or \
\begin{tikzpicture}
\node(0)at(1,-1){$x$};
\node(1)at(0,1){$y$};
\node(2)at(2,1){$z$};
\draw[d](0)--(1);
\draw[t](0)--(2);
\end{tikzpicture}
 \ or \
\begin{tikzpicture}
\node(0)at(1,-1){$x$};
\node(1)at(1,1){$y$};
\node(2)at(1,3){$z$};
\draw[d](0)--(1);
\draw[d](1)--(2);
\end{tikzpicture}
 \ or \
\begin{tikzpicture}
\node(0)at(1,-1){$x$};
\node(1)at(1,3){$y$};
\node(2)at(1,1){$z$};
\draw[d](0)--(2);
\draw[d](2)--(1);
\end{tikzpicture}
\end{center}
\caption{The (Pforest) axiom. The partial order $\le$ and the preorder $P$ are denoted by solid lines and dotted lines respectively.}\label{P-frame}
\end{figure}

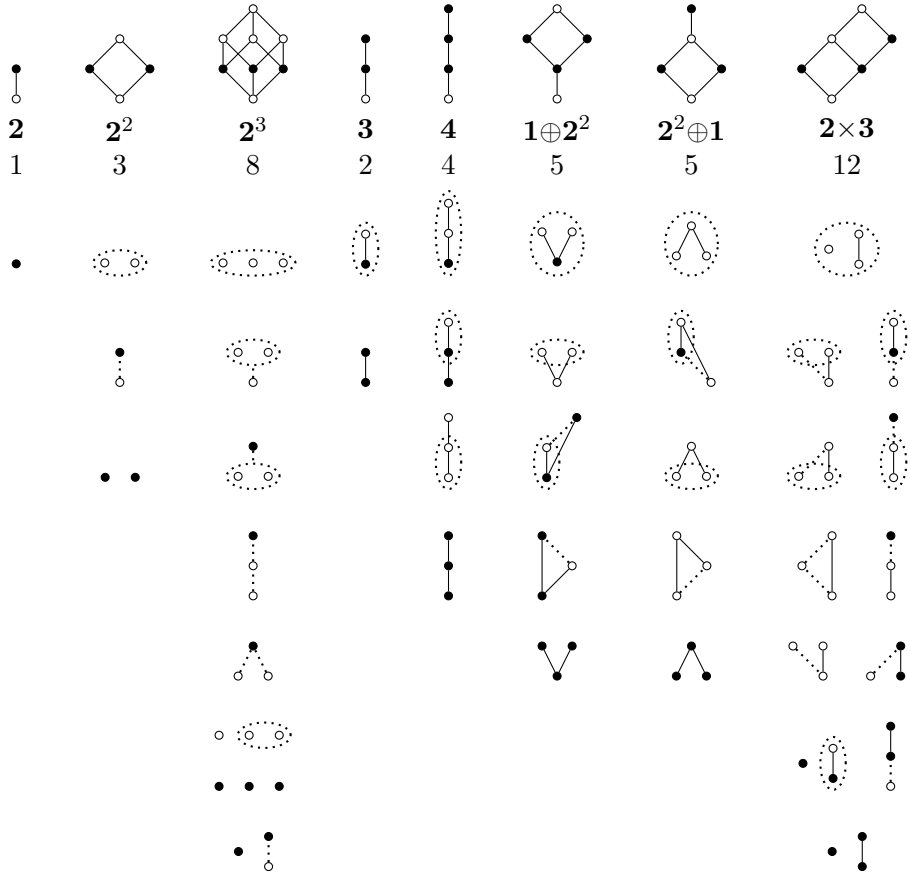
\begin{figure}[ht]
	\centering
	\tikzstyle{every picture} = [scale=.4]
	\tikzstyle{every node} = [draw, fill=white, ellipse, inner sep=0pt, minimum size=3pt]
	\tikzstyle{n} = [draw=none, rectangle, inner sep=2pt] 
	\tikzstyle{i} = [draw, fill=black, circle, inner sep=0pt, minimum size=3pt]

	\tikzstyle{d} = [thick, dotted ]

	\begin{tabular}{cccccccc}
		\begin{tikzpicture}
		\node at (0,-1)[n]{$\m 2$};
		\node at (0,-2.2)[n]{$1$};
		\node(1) at (0,1)[i]{};
		\node(0) at (0,0){};
		\draw(0) edge (1);
		\end{tikzpicture}
		&\quad
		\begin{tikzpicture}
		\node at (0,-1)[n]{$\m 2^2$};
		\node at (0,-2.2)[n]{$3$};
		\node(3) at (0,2){};
		\node(2) at (1,1)[i]{};
		\node(1) at (-1,1)[i]{};
		\node(0) at (0,0){};
		\draw(2) edge (3);
		\draw(1) edge (3);
		\draw(0) edge (1) edge (2);
		\end{tikzpicture}
		&\quad
		\begin{tikzpicture}
		\node at (0,-1)[n]{$\m 2^3$};
		\node at (0,-2.2)[n]{$8$};
		\node at (0,4)[n]{$ $};
		\node(7) at (0,3){};
		\node(6) at (1,2){};
		\node(5) at (0,2){};
		\node(4) at (-1,2){};
		\node(3) at (1,1)[i]{};
		\node(2) at (0,1)[i]{};
		\node(1) at (-1,1)[i]{};
		\node(0) at (0,0){};
		\draw(6) edge (7);
		\draw(5) edge (7);
		\draw(4) edge (7);
		\draw(3) edge (5) edge (6);
		\draw(2) edge (4) edge (6);
		\draw(1) edge (4) edge (5);
		\draw(0) edge (1) edge (2) edge (3);
		\end{tikzpicture}
		&\quad
		\begin{tikzpicture}
		\node at (0,-1)[n]{$\m 3$};
		\node at (0,-2.2)[n]{$2$};
		\node(2) at (0,2)[i]{};
		\node(1) at (0,1)[i]{};
		\node(0) at (0,0){};
		\draw(1) edge (2);
		\draw(0) edge (1);
		\end{tikzpicture}
		&\quad
		\begin{tikzpicture}
		\node at (0,-1)[n]{$\m 4$};
		\node at (0,-2.2)[n]{$4$};
		\node(3) at (0,3)[i]{};
		\node(2) at (0,2)[i]{};
		\node(1) at (0,1)[i]{};
		\node(0) at (0,0){};
		\draw(2) edge (3);
		\draw(1) edge (2);
		\draw(0) edge (1);
		\end{tikzpicture}
		&\quad
		\begin{tikzpicture}
		\node at (0,-1)[n]{$\m 1{\oplus}\m 2^2$};
		\node at (0,-2.2)[n]{$5$};
		\node(4) at (0,3){};
		\node(3) at (1,2)[i]{};
		\node(2) at (-1,2)[i]{};
		\node(1) at (0,1)[i]{};
		\node(0) at (0,0){};
		\draw(3) edge (4);
		\draw(2) edge (4);
		\draw(1) edge (2) edge (3);
		\draw(0) edge (1);
		\end{tikzpicture}
		&\quad
		\begin{tikzpicture}
		\node at (0,-1)[n]{$\m 2^2{\oplus}\m 1$};
		\node at (0,-2.2)[n]{$5$};
		\node(4) at (0,3)[i]{};
		\node(3) at (0,2){};
		\node(2) at (1,1)[i]{};
		\node(1) at (-1,1)[i]{};
		\node(0) at (0,0){};
		\draw(3) edge (4);
		\draw(2) edge (3);
		\draw(1) edge (3);
		\draw(0) edge (1) edge (2);
		\end{tikzpicture}
		&\quad
		\begin{tikzpicture}
		\node at (0.5,-1)[n]{$\m 2{\times}\m 3$};
		\node at (0.5,-2.2)[n]{$12$};
		\node(5) at (1,3){};
		\node(4) at (2,2)[i]{};
		\node(3) at (0,2){};
		\node(2) at (1,1)[i]{};
		\node(1) at (-1,1)[i]{};
		\node(0) at (0,0){};
		\draw(4) edge (5);
		\draw(3) edge (5);
		\draw(2) edge (3) edge (4);
		\draw(1) edge (3);
		\draw(0) edge (1) edge (2);
		\end{tikzpicture}
		\\
		\begin{tikzpicture}
		\node(0) at (0,1)[n]{};
		\node(1) at (0,1.25)[i]{};
		\end{tikzpicture}
		&\quad
		\begin{tikzpicture}
		\node(2) at (1,1){};
		\node(1) at (0,1){};
		\draw[d] (0.5,1) ellipse [x radius=25pt, y radius=12pt];
		\end{tikzpicture}
		&\quad
		\begin{tikzpicture}
		\node(3) at (2,1){};
		\node(2) at (1,1){};
		\node(1) at (0,1){};
		\draw[d] (1,1) ellipse [x radius=40pt, y radius=12pt];
		\end{tikzpicture}
		&\quad
		\begin{tikzpicture}
		\node(2) at (0,2){};
		\node(1) at (0,1)[i]{};
		\draw(1) edge (2);
		\draw[d] (0,1.5) ellipse [x radius=12pt, y radius=25pt];
		\end{tikzpicture}
		&\quad
		\begin{tikzpicture}
		\node(3) at (0,3){};
		\node(2) at (0,2){};
		\node(1) at (0,1)[i]{};
		\draw(2) edge (3);
		\draw(1) edge (2);
		\draw[d] (0,2) ellipse [x radius=12pt, y radius=40pt];
		\end{tikzpicture}
		&\quad
		\begin{tikzpicture}
		\node(3) at (.5,2){};
		\node(2) at (-.5,2){};
		\node(1) at (0,1)[i]{};
		\draw(1) edge (2) edge (3);
		\draw[d] (0,1.6) ellipse [x radius=25pt, y radius=30pt];
		\end{tikzpicture}
		&\quad
		\begin{tikzpicture}
		\node(4) at (0,2){};
		\node(2) at (0.5,1){};
		\node(1) at (-0.5,1){};
		\draw(2) edge (4);
		\draw(1) edge (4);
		\draw[d] (0,1.4) ellipse [x radius=25pt, y radius=30pt];
		\end{tikzpicture}
		&\quad
		\begin{tikzpicture}
		\node(4) at (1,2){};
		\node(2) at (1,1){};
		\node(1) at (0,1.5){};
		\draw(2) edge (4);
		\draw[d] (0.6,1.5) ellipse [x radius=30pt, y radius=25pt];
		\end{tikzpicture}
		\\
		&\quad
		\begin{tikzpicture}
		\node(2) at (0,2)[i]{};
		\node(1) at (0,1){};
		\draw[d](1) edge (2);
		\end{tikzpicture}
		&\quad
		\begin{tikzpicture}
		\node(4) at (1,2){};
		\node(3) at (0.5,1.8)[n]{};
		\node(2) at (0.5,1){};
		\node(1) at (0,2){};
		\draw[d] (2) edge (3);
		\draw[d] (0.5,2) ellipse [x radius=25pt, y radius=12pt];
		\end{tikzpicture}
		&\quad
		\begin{tikzpicture}
		\node(2) at (0,2)[i]{};
		\node(1) at (0,1)[i]{};
		\draw(1) edge (2);
		\end{tikzpicture}
		&\quad
		\begin{tikzpicture}
		\node(3) at (0,3){};
		\node(2) at (0,2)[i]{};
		\node(1) at (0,1)[i]{};
		\draw(2) edge (3);
		\draw(1) edge (2);
		\draw[d] (0,2.5) ellipse [x radius=12pt, y radius=25pt];
		\end{tikzpicture}
		&\quad
		\begin{tikzpicture}
		\node(3) at (.5,2){};
		\node(2) at (-.5,2){};
		\node(1) at (0,1){};
		\draw(1) edge (2) edge (3);
		\draw[d] (0,2) ellipse [x radius=25pt, y radius=12pt];
		\end{tikzpicture}
		&\quad
		\begin{tikzpicture}
		\node(4) at (0,2){};
		\node(2) at (0,1)[i]{};
		\node(1) at (1,0){};
		\draw(2) edge (4);
		\draw(1) edge (4);
		\draw[d](1) edge (2);
		\draw[d] (0,1.5) ellipse [x radius=12pt, y radius=25pt];
		\end{tikzpicture}
		&\quad
		\begin{tikzpicture}
		\node(4) at (1,2){};
		\node(2) at (1,1){};
		\node(1) at (0,2){};
		\draw(2) edge (4);
		\draw[d](1) edge (2);
		\draw[d] (0.5,2) ellipse [x radius=25pt, y radius=12pt];
		\end{tikzpicture}
		\quad
		\begin{tikzpicture}
		\node at (0,4)[n]{$ $};
		\node(3) at (0,3){};
		\node(2) at (0,2)[i]{};
		\node(1) at (0,1){};
		\draw(2) edge (3);
		\draw[d](1) edge (2);
		\draw[d] (0,2.5) ellipse [x radius=12pt, y radius=25pt];
		\end{tikzpicture}
		\\
		&\quad
		\begin{tikzpicture}
		\node(0) at (0,1)[n]{};
		\node(2) at (1,1.25)[i]{};
		\node(1) at (0,1.25)[i]{};
		\end{tikzpicture}
		&\quad
		\begin{tikzpicture}
		\node(4) at (1,1){};
		\node(3) at (0.5,1.2)[n]{};
		\node(2) at (0.5,2)[i]{};
		\node(1) at (0,1){};
		\draw[d] (2) edge (3);
		\draw[d] (0.5,1) ellipse [x radius=25pt, y radius=12pt];
		\end{tikzpicture}
		&\quad
		&\quad
		\begin{tikzpicture}
		\node(3) at (0,3){};
		\node(2) at (0,2){};
		\node(1) at (0,1){};
		\draw(2) edge (3);
		\draw(1) edge (2);
		\draw[d] (0,1.5) ellipse [x radius=12pt, y radius=25pt];
		\end{tikzpicture}
		&\quad
		\begin{tikzpicture}
		\node(3) at (1,3)[i]{};
		\node(2) at (0,2){};
		\node(1) at (0,1)[i]{};
		\draw(1) edge (2) edge (3);
		\draw[d](2) edge (3);
		\draw[d] (0,1.5) ellipse [x radius=12pt, y radius=25pt];
		\end{tikzpicture}
		&\quad
		\begin{tikzpicture}
		\node at (0,3)[n]{$ $};
		\node(4) at (0,2){};
		\node(2) at (0.5,1){};
		\node(1) at (-0.5,1){};
		\draw(2) edge (4);
		\draw(1) edge (4);
		\draw[d] (0,1) ellipse [x radius=25pt, y radius=12pt];
		\end{tikzpicture}
		&\quad
		\begin{tikzpicture}
		\node(4) at (1,2){};
		\node(2) at (1,1){};
		\node(1) at (0,1){};
		\draw(2) edge (4);
		\draw[d](1) edge (4);
		\draw[d] (0.5,1) ellipse [x radius=25pt, y radius=12pt];
		\end{tikzpicture}
		\quad
		\begin{tikzpicture}
		\node at (0,3.5)[n]{$ $};
		\node(3) at (0,3)[i]{};
		\node(2) at (0,2){};
		\node(1) at (0,1){};
		\draw[d](2) edge (3);
		\draw(1) edge (2);
		\draw[d] (0,1.5) ellipse [x radius=12pt, y radius=25pt];
		\end{tikzpicture}
		\\
		&\quad
		&\quad
		\begin{tikzpicture}
		\node at (0,3.5)[n]{$ $};
		\node(3) at (0,3)[i]{};
		\node(2) at (0,2){};
		\node(1) at (0,1){};
		\draw[d](2) edge (3);
		\draw[d](1) edge (2);
		\end{tikzpicture}
		&\quad
		&\quad
		\begin{tikzpicture}
		\node(3) at (0,3)[i]{};
		\node(2) at (0,2)[i]{};
		\node(1) at (0,1)[i]{};
		\draw(2) edge (3);
		\draw(1) edge (2);
		\end{tikzpicture}
		&\quad
		\begin{tikzpicture}
		\node(3) at (1,2){};
		\node(2) at (0,3)[i]{};
		\node(1) at (0,1)[i]{};
		\draw[d](3) edge (2);
		\draw(1) edge (2) edge (3);
		\end{tikzpicture}
		&\quad
		\begin{tikzpicture}
		\node at (0,3.5)[n]{$ $};
		\node(3) at (0,3){};
		\node(2) at (1,2){};
		\node(1) at (0,1){};
		\draw(2) edge (3);
		\draw[d](1) edge (2);
		\draw(1) edge (3);
		\end{tikzpicture}
		&\quad
		\begin{tikzpicture}
		\node at (0,3.5)[n]{$ $};
		\node(3) at (0,3){};
		\node(2) at (-1,2){};
		\node(1) at (0,1){};
		\draw[d](2) edge (3);
		\draw[d](1) edge (2);
		\draw(1) edge (3);
		\end{tikzpicture}
		\quad \
		\begin{tikzpicture}
		\node at (0,4)[n]{$ $};
		\node(3) at (0,3)[i]{};
		\node(2) at (0,2){};
		\node(1) at (0,1){};
		\draw[d](2) edge (3);
		\draw(1) edge (2);
		\end{tikzpicture}
		\\
		&\quad
		&\quad
		\begin{tikzpicture}
		\node(4) at (0,2)[i]{};
		\node(2) at (0.5,1){};
		\node(1) at (-0.5,1){};
		\draw[d](2) edge (4);
		\draw[d](1) edge (4);
		\end{tikzpicture}

		&\quad
		&\quad
		&\quad
		\begin{tikzpicture}
			\node(3) at (.5,2)[i]{};
			\node(2) at (-.5,2)[i]{};
			\node(1) at (0,1)[i]{};
			\draw(1) edge (2) edge (3);
		\end{tikzpicture}
		&\quad
		\begin{tikzpicture}
		\node at (0,3)[n]{$ $};
		\node(4) at (0,2)[i]{};
		\node(2) at (0.5,1)[i]{};
		\node(1) at (-0.5,1)[i]{};
		\draw(2) edge (4);
		\draw(1) edge (4);
		\end{tikzpicture}
		&\quad
		\begin{tikzpicture}
		\node(4) at (1,2){};
		\node(2) at (1,1){};
		\node(1) at (0,2){};
		\draw(2) edge (4);
		\draw[d](1) edge (2);
		\end{tikzpicture}
		\quad
		\begin{tikzpicture}
		\node(4) at (1,2)[i]{};
		\node(2) at (1,1)[i]{};
		\node(1) at (0,1){};
		\draw(2) edge (4);
		\draw[d](1) edge (4);
		\end{tikzpicture}
		\\
		&\quad
		&\quad
		\begin{tikzpicture}
		\node(3) at (2,2.7){};
		\node(2) at (1,2.7){};
		\node(1) at (0,2.7){};
		\draw[d] (1.5,2.7) ellipse [x radius=25pt, y radius=12pt];
		\node(3) at (2,1)[i]{};
		\node(2) at (1,1)[i]{};
		\node(1) at (0,1)[i]{};
		\end{tikzpicture}
		&\quad
		&\quad
		&\quad
		&\quad
		&\quad
		\begin{tikzpicture}
		\node at (0,3)[n]{$ $};
		\node(4) at (1,2){};
		\node(2) at (1,1)[i]{};
		\node(1) at (0,1.5)[i]{};
		\draw(2) edge (4);
		\draw[d] (1,1.5) ellipse [x radius=12pt, y radius=25pt];
		\end{tikzpicture}
		\quad
		\begin{tikzpicture}
		\node at (0,4)[n]{$ $};
		\node(3) at (0,3)[i]{};
		\node(2) at (0,2)[i]{};
		\node(1) at (0,1){};
		\draw(2) edge (3);
		\draw[d](1) edge (2);
		\end{tikzpicture}
		\\
		&\quad
		&\quad
		\begin{tikzpicture}
		\node at (0,3)[n]{$ $};
		\node(4) at (1,2)[i]{};
		\node(2) at (1,1){};
		\node(1) at (0,1.5)[i]{};
		\draw[d] (2) edge (4);
		\end{tikzpicture}
		&\quad
		&\quad
		&\quad
		&\quad
		&\quad
		\begin{tikzpicture}
		\node at (0,3)[n]{$ $};
		\node(4) at (1,2)[i]{};
		\node(2) at (1,1)[i]{};
		\node(1) at (0,1.5)[i]{};
		\draw(2) edge (4);
		\end{tikzpicture}
	\end{tabular}
	\caption{All 40 preorder forest $P$-frames $(W,\le,P)$ with up to 3 elements. Solid lines show $(W,\le)$, dotted lines show the additional edges of $P$, and the identity (if it exists) is the set of black dots. The first row shows the lattice of downsets, and the  Boolean quantales from \cite{AJ2020} appear in the first three columns.}\label{Pframes}
\end{figure}

We are now ready to state the main result. We use the algebraic characterization of associativity in Lemma~\ref{pqprops}.

\begin{thm}\label{assoc}
Let $\m W=(W,\leq,P)$ be a preorder forest $P$-frame and $\m D(\m W)$ its corresponding downset algebra.
Then the operation $x\ncdot  y=(\p x\wedge y)\vee(x\wedge \p y)$
is associative in $\m D(\m W)$.
\end{thm}
\begin{proof}
Let $\m W=(W,\le,P)$ be a preorder forest $P$-frame and $\m D(\m W)$ its $d\ell \p$-algebra of downsets with operator $\p $. 
Since $P$ is a preorder, $\m D(\m W)$ is a $d\ell \p$-closure algebra.
By Lemma~\ref{five}, a $d\ell \p$-closure algebra
is associative if and only if $\p(x)\wedge \p(y)\le \p(\p(x)\wedge y)\vee(x\wedge \p(y))$.
By Lemma~\ref{frameprops} this is equivalent to the frame property
\[\qquad xPy\ \&\ xPz\Rightarrow \exists w(xPw\ \&\ (wPy\ \&\ w\le z\text{ or }w\le y\ \&\ wPz)).\qquad (*)\]

We now show that this frame property holds in $\m W$. We know that $P$ is reflexive and (Pforest) holds.

Assume $xPy$ and $xPz$. By (Pforest) there are four cases:
\begin{enumerate}
\item $x\le y$: take $w=x$. Then $xPx$, $x\le y$ and $xPz$, hence $(*)$
holds.
\item $x\le z$: again take $w=x$. Then the other disjunct of $(*)$
holds.
\item $yPz$: take $w=y$. Then $xPy$, $y\le y$ and $yPz$, hence $(*)$ holds.
\item $zPy$: take $w=z$. Then $xPz$, $zPy$ and $y\le y$, hence again $(*)$ holds.\qed
\end{enumerate}
\let\qed\relax
\end{proof}

The universal class of preorder forest $P$-frames is strictly contained in the 
class of all $P$-frames in which $x\ncdot  y$ is associative. In fact the latter
class is not closed under substructures, hence not a universal class: $W=\{0,1,2,3\}$, ${\le}=id_W\cup\{(0,1),(0,2),(0,3)\}$, $P={\le}\cup\{(1,0),(1,2),(1,3)\}$ is a $P$-frame with associative $\cdot$ (use e.g. Lemma~\ref{five}), but restricting $\le, P$ to the subset $\{1,2,3\}$ gives a $P$-frame where $\cdot$ fails to be associative, hence (Pforest) also fails.

A \emph{$d\ell$-semilattice} is an associative commutative idempotent $d\ell$-magma. The point of the previous result is that it allows the construction of perfect associative commutative idempotent $d\ell$-magmas and idempotent bunched implication algebras from preorder forest $P$-frames. This is much simpler than constructing the ternary relation $R$ of the Birkhoff frame of such algebras. For example the Hasse diagrams for all the preorder forest $P$-frames with up to 3 elements are shown in Figure~\ref{Pframes}, with the preorder $P$ given by dotted lines and ovals. The corresponding ternary relations can be calculated from $P$, but would have been hard to include in each diagram.

We now examine when a $P$-frame will have an identity element.

\begin{lem}\label{E}
Let $\m W$ be a $P$-frame and $E$ a downset of $W$.
Then the downset algebra $D(\m W)$ has $E$ as identity element for $\cdot$ if and only if $E=\{x\in W \mid \forall y(xPy\Rightarrow x\le y)\}$ and $\p E=W$.
\end{lem}
\begin{proof}
In the forward direction assume a downset $E$ is the identity for $\cdot$, and let $y\in W$. 
It follows from Lemma~\ref{pqprops}(5) that $\p E = W$ since $W$ is the top element in $D(\m W)$, and moreover, $(\p E\cap\da y)\cup(E\cap\p(\da y))=\da y$. Hence $E\cap\p(\da y)\subseteq \da y$ for all $y$, which shows that if $x\in E$ then $\forall y(xPy\Rightarrow x\le y)$ holds. Now let $x\in W$ satisfy $\forall y(xPy\Rightarrow x\le y)$. From $p E=W$ we deduce that $xPz$ for some $z$, hence $x\le z$ and, since $E$ is a downset, $x\in E$.

Conversely, by the definition of $E$, if $x \in E$, then $xPy \Rightarrow x \le y$ holds for all $y\in W$. Hence by Lemma~\ref{frameprops}(5) for all $X\in D(\m W)$ we have $\p X\cap E\subseteq X$.
Since $\p E=W$ together with Lemma~\ref{pqprops}(5), it follows that $E$ is an identity element in the downset algebra.
\end{proof}

\section{Weakly conservative perfect \texorpdfstring{$d\ell$}{dl}-magmas and Birkhoff frames}

In this section we explore a special case that arises when the relations $P$ and $Q$ are
determined from $R$ by $xPy\Leftrightarrow xRyx$ and $xQy\Leftrightarrow xRxy$, i.e., the existential quantifier from the previous section is instantiated by $z=x$. We first discuss some related algebraic properties.

A binary operation $\cdot$ is called \emph{conservative} (or \emph{quasitrivial}) if the output value is always one of the two inputs, i.e., it satisfies $xy=x\text{ or }xy=y\text{ for all }x,y \in A$. Note that this property implies idempotence.

In general a $d\ell$-magma is idempotent if and only if it satisfies $x\mt y\le xy\le x\jn y$, since
$
x\mt y=(x\mt y)(x\mt y)\le xy\le (x\jn y)(x\jn y)=x\jn y,
$
and conversely, identifying $x,y$ we have $x\le xx\le x$.

A perfect $\ell$-magma $A$ is called \emph{weakly conservative} if it satisfies the formula
\[xy=x\mt y\text{ or }xy=x\text{ or }xy=y\text{ or }xy=x\jn y\text{ for all $x,y\in J(A)$}.\]
So for completely join-irreducible elements $x,y$ the product
$xy\in\{x\wedge y,x,y,x\vee y\}$.
This is a generalization of conservativity in two ways since there are additional possibilities for
the value of $xy$ and the formula only needs to hold for completely join-irreducible elements.

A typical example of a weakly conservative perfect $\ell$-magma is an atomic Boolean algebra with an idempotent binary operation $xy$. In this case the completely join-irreducible elements are the atoms of the Boolean algebra, and for any two atoms $x,y$ the interval $[x\wedge y,x\vee y]\subseteq \{x\wedge y,x,y,x\vee y\}$. Since we observed previously that $x\wedge y\le xy\le x\vee y$ it follows that $xy$ can only take on one of the four values $x\wedge y,x,y,x\vee y$.

The notation $x\le yRzw$ is shorthand for $x\le y$ and $yRzw$. We also write $x\le y,z$ as an abbreviation for $x\le y$ and $x\le z$.
A Birkhoff frame is called \emph{weakly conservative} if it satisfies
\[
xRyz\Leftrightarrow x\leq y,z \text{ or } x\leq yRyz\text{ or } x\leq zRyz.
\]
This terminology is motivated by the following result.
\begin{lem}
Let $\m{W}$ be a Birkhoff frame. Then $D(\m{W})$ is weakly conservative if and only if $\m W$ is weakly conservative.
\end{lem}
\begin{proof} We first note that weak conservativity for $D(\m W)$ can be written in conjunctive form as
\begin{equation}
x\wedge y\le xy\le x\vee y\text{ and }(xy\le x\text{ or }y\le xy)\text{ and }(xy\le y\text{ or }x\le xy).
\end{equation}
Likewise weak conservativity for $\m W$ in conjunctive form (on right hand side) is
\begin{equation}
wRxy\Leftrightarrow(w\le x\text{ or }w\le y)\text{ and }(w\le x\text{ or }yRxy)\text{ and }(w\le y\text{ or }xRxy).
\end{equation}
Now assume (1) and for $w,x,y\in W$ assume $wRxy$. To simplify notation, we identify elements of $W$ with their principal filters in $D(\m W)$. Since $w$ is join-irreducible and $xy\le x\vee y$, it follows that $w\le x$ or $w\le y$. Next, to prove that $w\le x\text{ or }yRxy$,
assume $w\nleq x$. Then $xy\nleq x$, and again $xy\le x\vee y$ implies $xy\le y$. The last conjunct is proved similarly.
Suppose now that the right hand side of (2) holds for $w,x,y\in W$. Using the original disjunctive form, there are 3 cases: if $w\le x,y$ then by (1) $w=w\wedge w\le ww\le xy$,
hence $wRxy$. If $w\le x\le Rxy$ then $w\le x\le xy$, so again we obtain $wRxy$. The third case is similar.

Conversely, assume (2) holds and let $x,y$ be join-irreducibles of $D(\m W)$. To see that $x\wedge y\le xy$, let $w$ be any join-irreducible such that $w\le x\wedge y$, in which case $wRxy$ follows from the disjunctive form of (2), hence $w\le xy$. Since weak conservativity of $D(\m W)$ implies idempotence, $xy\le x\vee y$ follows from Lemma~\ref{idem}. To prove that $xy\le x$ or $y\le xy$, assume $y\nleq xy$, whence $yRxy$ does not hold. For any join-irreducible $w\le xy$ we have $wRxy$, and since (2) implies $w\le x$ or $yRxy$, we conclude
that $xy\le x$. Finally, the conjunct $xy\le y$ or $x\le xy$ follows by symmetry of $x,y$.
\end{proof}

Next we show that in every weakly conservative Birkhoff frame the ternary relation
$R$ is determined by two binary relations $P,Q$ defined by $xPy\Leftrightarrow xRyx$ and $xQy\Leftrightarrow xRxy$. This is simpler than the previous definitions with an existential quantifier, but they need not be weakening relations, hence they do not produce a $PQ$-frame. Instead they are axiomatized by the following conditions.

A \emph{PQ-structure} is of the form $(W,\leq,P,Q)$ where $(W,\leq)$ is a poset and
\begin{center}
\begin{tabular}{ll}
(P0) $x\leq y \implies  xPy$&
(Q0) $x\leq y \implies  xQy$\\
(P1) $x\le y ~\&~ xPz \implies x\le z \text{ or } yPz\quad$&
(Q1) $x\le y ~\&~ xQz \implies x\le z \text{ or } yQz$\\
(P2) $xPy\leq z \implies xPz$&
(Q2) $xQy\leq z \implies xQz$.
\\
\end{tabular}
\end{center}
Note that (P0) and (Q0) together with reflexivity of $\le$ imply that both $P$ and $Q$ are reflexive.
The following result shows that $PQ$-structures and weakly conservative Birkhoff frames are
definitionally equivalent.
This generalizes an earlier result of \cite{AJ2020} where the partial order $\leq$ was assumed to be $=$.

\begin{thm}\label{PQwc} \hfill
\begin{enumerate}[label=\textup{(\arabic*)}]
	\item 
	For a $PQ$-structure $(W,\leq, P,Q)$ let $xRyz$ be defined by
	$x\leq y,z \text{ or } x\leq yQz$ or $x\leq zPy$. Then $(W,\leq,R)$ is a weakly conservative Birkhoff frame and $P,Q$ are recovered via
	$xPy\Leftrightarrow xRyx~\text{ and }~xQy\Leftrightarrow xRxy$.

	\item For a weakly conservative Birkhoff frame $(W,\leq,R)$ define
	$xPy\Leftrightarrow xRyx$ and $xQy\Leftrightarrow xRxy$. Then $(W,\leq,P,Q)$ is a $PQ$-structure and 
$xRyz\Leftrightarrow x\leq y, z \text{ or } x\leq yQz\text{ or } x\leq zPy$.
        \end{enumerate}
\end{thm}
\begin{proof}
(1) Assume $(W,\le,P,Q)$ is a $PQ$-structure and let $R$ be defined as above. We need to prove that $R$ satisfies (R1), (R2) and (R3).

(R1) Assume $xRyz$ and $w\le x$. Then $wRyz$ follows from (the expanded form of) $xRyz$ by transitivity of $\le$.

(R2) Assume $xRyz$ and $y\le w$.
By assumption $x\leq y, z \text{ or } x\leq yQz\text{ or } x\leq zPy$, so we have 3 subcases.
In the first subcase $x\leq y$ implies $x\leq w$ by transitivity of $\le$, hence $x\le w,z$.

In the second subcase $x\leq yQz$. We also have $y\le w$, hence $x\le w$.
A substitution instance of (Q1) is $yQz ~\&~ y\leq w \Rightarrow y\le z \text{ or } wQz$,
hence $y\le z \text{ or } wQz$. From the assumption that $x\le y$ it follows that $x\le z \text{ or } wQz$. Since $x\le w$ holds we obtain $x\le w,z \text{ or } x\le wQz$.
Therefore $xRwz$ holds.
In the third subcase $x\leq z Py$ implies $x\leq zPw$ by (P2), and again $xRyz$ holds.

(R3) The argument is symmetric to the one for (R2).

To prove that $R$ is weakly conservative, we show that $xPy\Leftrightarrow xRyx$ and $xQy\Leftrightarrow xRxy$.
Now $xRyx$ is equivalent to $x\leq y, x \text{ or } x\leq yQx\text{ or } x\leq xPy$ which simplifies to $x\le y \text{ or } xPy$, and by (P0) this is equivalent to $xPy$.
Similarly $xRxy$ is equivalent to $xQy$.

(2) Assume $(W,\le,R)$ is a weakly conservative Birkhoff frame and define $xPy\Leftrightarrow xRyx$ and $xQy\Leftrightarrow xRxy$. We show that (P0--P2) hold, and the arguments for (Q0--Q2) are similar.
Note that weak conservativity of $R$ is equivalent to $xRyz\Leftrightarrow x\leq y, z \text{ or } x\leq yQz\text{ or } x\leq zPy.$

(P0) Assume $x\le y$. From reflexivity of $\le$ it follows that $x\le y,x$. This implies that $xRyx$ holds, and hence $xPy$.

(P1) Assume $x\le y$ and $xPz$. Then $xRzx$ holds, and (R3) implies $xRzy$, or equivalently
$x\leq z, y \text{ or } x\leq zQy \text{ or } x\leq yPz$. Since $x\le y$, this disjunction simplifies
to $x\leq z \text{ or } x\leq zQy \text{ or } yPz$. Since $x\leq z$ is a conjunct of the middle part,
the formula simplifies to $x\le z$ or $yPz$.

(P2) Assume $xPy\le z$. This is equivalent to $xRyx$ and $y\le z$, so by (R2) $xRzx$ follows. This is equivalent to $x\leq z, x \text{ or } x\leq zQx\text{ or } x\leq xPz$, which simplifies by reflexivity of $\le$ to $x\le z$ or $xPz$, and further by (P0) to $xPz$.
\end{proof}

Conditions (P2) and (Q2) ensure that $P,Q$ are ``half-weakening relations''.
Hence a $PQ$-structure is a $PQ$-frame if and only if it satisfies the other half
\begin{center}
\begin{tabular}{ll}
(P2') $x\le yPz \implies xPz\qquad\qquad$&(Q2') $x\le yQz \implies xQz$.
\end{tabular}
\end{center}
It follows from Theorems~\ref{equivframes} and \ref{PQwc} that in a $PQ$-structure that is also a $PQ$-frame, the weakly conservative ternary relation $R$ can be defined in two equivalent ways: as $(xPy\ \&\ x\le z)\text{ or }(x\le y\ \&\ xQz)$ and as $x\leq y,z \text{ or } x\leq yQz$ or $x\leq zPy$.

\section{Counting preorder forests and linear \texorpdfstring{$P$}{P}-frames}

In the case when the poset $(W,\le)$ is an antichain, a preorder forest $P$ is simply a preorder $P\subseteq W^2$ such that $xPy$ and $xPz$ implies $yPz$ or $zPy$.
A \emph{preorder tree} is a connected component of a preorder forest.
A \emph{rooted} preorder forest is defined to have an equivalence class of $P$-maximal elements in each component. For finite preorder forests this is always the case. Let $F_n$ denote the number of preorder forests and $T_n$ the number of preorder trees with $n$ elements (up to isomorphism). We also let $F_0=1$.

A preorder forest \emph{has singleton roots} if the $P$-maximal equivalence class of each component is a singleton set. The number of preorder forests and trees with singleton roots is denoted by $F^s_n$ and $T^s_n$ respectively.

Note that every preorder forest gives rise to a unique preorder tree with a singleton root by adding one new element $r$ such that for all $x\in W$ we have
$xPr$. It follows that $T^s_n=F_{n-1}$.

\begin{table}
\begin{center}\tabcolsep3pt
\begin{tabular}{|rccccccc|}\hline
cardinality $n=$        & 1 & 2 & 3 & 4 & 5 & 6 & 7 \\ \hline
preorder trees $T_n=$   & 1 & 2 & 5 & 13 & 37 & 108 & 337\\
               $c_n=$   & 1 & 5 &16 & 57 &186 & 668 &\\
preorder forests $F_n=$ & 1 & 3 & 8 & 24 & 71 & 224 &\\
preorder trees with singleton roots $T^s_n=$  & 1 & 1 & 3 &  8 & 24 & 71 &224\\
               $c^s_n=$   & 1 & 3 & 10 & 35 & 121 & 438 &\\
preorder forests with singleton roots $F^s_n=$& 1 & 2 & 5 & 14 & 41 & 127 &\\ \hline
\end{tabular}
\bigskip
\end{center}
\caption{Number of preorder trees and forests (up to isomorphism)}
\label{fig:table}
\end{table}

Every preorder tree with a non-singleton root equivalence class
and $n$ elements is obtained from a preorder tree with $n-1$ elements by adding one more element to the root equivalence class. Hence for $n>0$ we have $T_n=F_{n-1}+T_{n-1}$. The Euler transform of $T_n$ is used to calculate the
next value of $F_n$ as follows:
\begin{align*}
c_n&=\sum_{d|n}d\cdot T_n\qquad\qquad
F_n=\frac1n\sum_{k=1}^n c_k\cdot F_{n-k}.
\end{align*}
Since preorder forests with singleton roots are disjoint unions of preorder trees with singleton roots, $F^s_n$ is calculated by an Euler transform from $T^s_n$.

\begin{cor}
The sequence $F^s_n$ is the Euler transform of $T^s_n$.
\end{cor}

While it is difficult to count preorder forest $P$-frames in general, it is simple to count the linear ones. Note that the (Pforest) axiom is actually redundant for linearly ordered $P$-frames.

\begin{thm}\label{chains}
There are $2^{n-1}$ linearly ordered forest $P$-frames. In the algebraic setting, for $n>1$, there are $2^{n-2}$ unary-determined commutative doubly idempotent linear semirings with $n$ elements, and $n-1$ of them have an identity element.
\end{thm}
\begin{proof}
Let $\m W$ be a linearly ordered $P$-frame with elements $W=\{1<2<\dots <n\}$ such that $P$ is transitive and (P0) holds. Then each possible relation $P$ on $W$ is determined by choosing a subset $S$ of the edges $\{(2,1), (3,2), \ldots,(n,n-1)\}$ and defining $P$ to be the transitive closure of $S\cup{\le}$. Since there are $n-1$ such edges to choose from, the number of $P$-frames is $2^{n-1}$. 

Let $\m A$ be a unary-determined commutative doubly idempotent linear semiring with $n$ elements. Then the $P$-frame $\m W$ associated with $\m A$ has $n-1$ elements, is linearly ordered, and $P$ is reflexive and transitive since $\cdot$ is idempotent and associative. Hence there are $2^{n-2}$ such algebras.

By Lemma~\ref{pqprops} such an algebra $\m A$ will have an identity $1$ if and only if the operator $\p $ in the corresponding d$\ell \p$-closure algebra satisfies the conditions $\p 1 = \top$ and $\p x \wedge 1 \leq x$ for every $x \in A$. The first condition means that $1$ is not closed (unless it is $\top$), and there are no closed elements other than $\top$ above $1$. Since the partial order is a linear order and $\p $ is inflationary, the second condition is equivalent to $\p x = x$ or $1 \leq x$. That is to say, $1$ is also the minimum non-closed element in $\m A$. Hence the $n$-element unary-determined commutative doubly idempotent linear semirings with identity are the chains with the identity element in the $k$-th position, where $1 < k \leq n$, with every element below $1$ closed and every element $\geq 1$ either non-closed or equal to $\top$. Such semirings are uniquely identified by the position of the identity element, which can never be $\bot$. There are $n-1$ possible positions, and hence $n-1$ semirings with an identity element.
\end{proof}

\section{Subdirectly irreducible \texorpdfstring{$d\ell \p$}{dlp}-algebras and unary-determined BI-chains}
Let $\mathcal V$ be a variety (= equational class) of unary-determined $d\ell$-magmas. Recall that an algebra $\m A$ is \emph{subdirectly irreducible} if its congruence lattice Con$\m A$ has a unique minimal nontrivial congruence, and $\m A$ is \emph{simple} if
Con$\m A$ has exactly two elements. By Birkhoff's subdirect representation theorem every algebra is (subdirectly) embedded in a product of subdirectly irreducible factors, hence $\mathcal V=\mathbb{ISP}(SI(\mathcal V))$ where $SI(\mathcal V)$ is the class of all subdirectly irreducible members of $\mathcal V$ and $\mathbb I, \mathbb S, \mathbb P$ are the class operators that return all isomorphic copies, all subalgebras and all products of members of their input class.

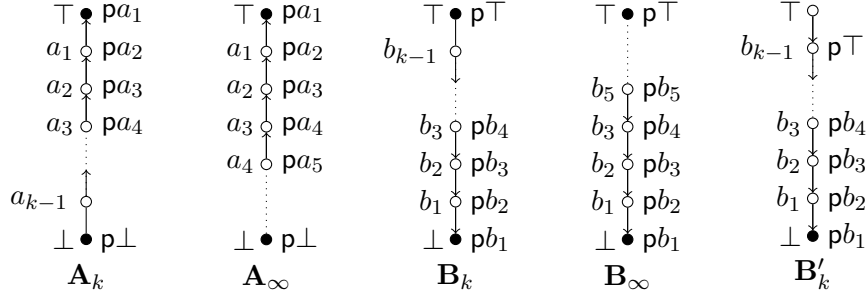
\begin{figure}
	\tikzstyle{every picture} = [scale=.5]
	\tikzstyle{every node} = [draw, fill=white, ellipse, inner sep=0pt, minimum size=4pt]
	\tikzstyle{n} = [draw=none, rectangle, inner sep=2pt] 
	\tikzstyle{c} = [draw, fill=black, circle, inner sep=0pt, minimum size=4pt]
    \centering
		\begin{tikzpicture}
		\node at (0,-1)[n]{$\m A_k$};
		\node(5)at(0,6)[c,label=left:$\top$,label=right:$\p a_1$]{};
		\node(4)at(0,5)[label=left:$a_{1}$,label=right:$\p a_2$]{}edge(5);
		\node(3)at(0,4)[label=left:$a_{2}$,label=right:$\p a_3$]{}edge(4);
		\node(2)at(0,3)[label=left:$a_{3}$,label=right:$\p a_4$]{}edge(3);
		\node(12)at(0,2)[n]{};
		\node(1)at(0,1)[label=left:$a_{k-1}$]{}edge[dotted](2);
		\node(0)at(0,0)[c,label=left:$\bot$,label=right:$\p\bot$]{}edge(1);
		\draw[->](1)--(12);
		\draw[->](2)--(3);
		\draw[->](3)--(4);
		\draw[->](4)--(5);
		\end{tikzpicture}
		\quad\quad
		\begin{tikzpicture}
		\node at (0,-1)[n]{$\m A_\infty$};
		\node(5)at(0,6)[c,label=left:$\top$,label=right:$\p a_1$]{};
		\node(4)at(0,5)[label=left:$a_{1}$,label=right:$\p a_2$]{}edge(5);
		\node(3)at(0,4)[label=left:$a_{2}$,label=right:$\p a_3$]{}edge(4);
		\node(2)at(0,3)[label=left:$a_{3}$,label=right:$\p a_4$]{}edge(3);
		\node(1)at(0,2)[label=left:$a_{4}$,label=right:$\p a_5$]{}edge(2);
		\node(0)at(0,0)[c,label=left:$\bot$,label=right:$\p\bot$]{}edge[dotted](1);
		\draw[->](1)--(2);
		\draw[->](2)--(3);
		\draw[->](3)--(4);
		\draw[->](4)--(5);
		\end{tikzpicture}
		\quad
		\begin{tikzpicture}
		\node at (0,-1)[n]{$\m B_k$};
		\node(5)at(0,6)[c,label=left:$\top$,label=right:$\p\top$]{};
		\node(4)at(0,5)[label=left:$b_{k-1}$]{}edge(5);
		\node(34)at(0,4)[n]{};
		\node(3)at(0,3)[label=left:$b_{3}$,label=right:$\p b_4$]{}edge[dotted](4);
		\node(2)at(0,2)[label=left:$b_{2}$,label=right:$\p b_3$]{}edge(3);
		\node(1)at(0,1)[label=left:$b_1$,label=right:$\p b_2$]{}edge(2);
		\node(0)at(0,0)[c,label=left:$\bot$,label=right:$\p b_1$]{}edge(1);
		\draw[->](4)--(34);
		\draw[->](3)--(2);
		\draw[->](2)--(1);
		\draw[->](1)--(0);
		\end{tikzpicture}
		\quad\quad
		\begin{tikzpicture}
		\node at (0,-1)[n]{$\m B_\infty$};
		\node(5)at(0,6)[c,label=left:$\top$,label=right:$\p\top$]{};
		\node(4)at(0,4)[label=left:$b_{5}$,label=right:$\p b_5$]{}edge[dotted](5);
		\node(3)at(0,3)[label=left:$b_{3}$,label=right:$\p b_4$]{}edge(4);
		\node(2)at(0,2)[label=left:$b_{2}$,label=right:$\p b_3$]{}edge(3);
		\node(1)at(0,1)[label=left:$b_1$,label=right:$\p b_2$]{}edge(2);
		\node(0)at(0,0)[c,label=left:$\bot$,label=right:$\p b_1$]{}edge(1);
		\draw[->](4)--(3);
		\draw[->](3)--(2);
		\draw[->](2)--(1);
		\draw[->](1)--(0);
		\end{tikzpicture}
		\quad
		\begin{tikzpicture}
		\node at (0,-1)[n]{$\m B'_k$};
		\node(5)at(0,6)[label=left:$\top$]{};
		\node(4)at(0,5)[label=left:$b_{k-1}$,label=right:$\p\top$]{}edge(5);
		\node(34)at(0,4)[n]{};
		\node(3)at(0,3)[label=left:$b_{3}$,label=right:$\p b_4$]{}edge[dotted](4);
		\node(2)at(0,2)[label=left:$b_{2}$,label=right:$\p b_3$]{}edge(3);
		\node(1)at(0,1)[label=left:$b_1$,label=right:$\p b_2$]{}edge(2);
		\node(0)at(0,0)[c,label=left:$\bot$,label=right:$\p b_1$]{}edge(1);
		\draw[->](5)--(4);
		\draw[->](4)--(34);
		\draw[->](3)--(2);
		\draw[->](2)--(1);
		\draw[->](1)--(0);
		\end{tikzpicture}
    \caption{Subdirectly irreducible $d\ell \p$-chains (black elements satisfy $\p x=x$).}
    \label{fig:sidlpch}
\end{figure}

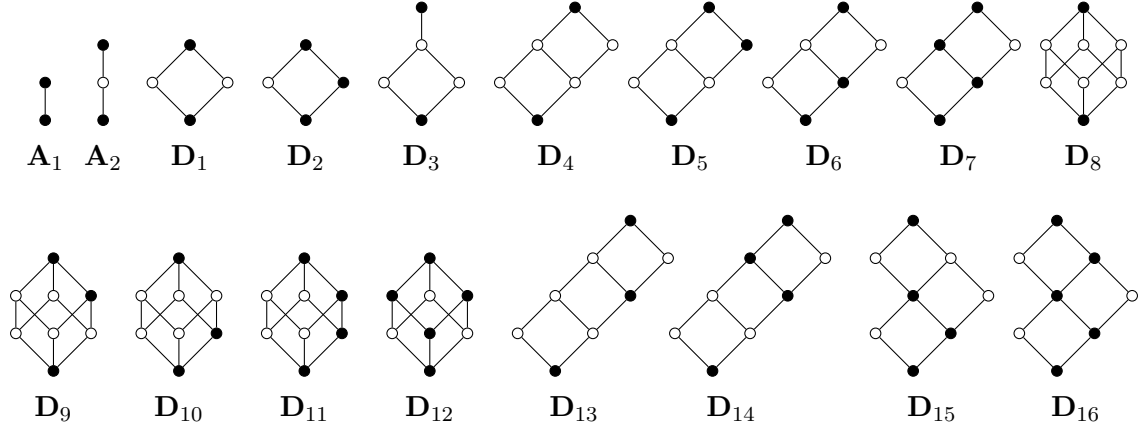
\begin{figure}
	\tikzstyle{every picture} = [scale=.5]
	\tikzstyle{every node} = [draw, fill=white, ellipse, inner sep=0pt, minimum size=4pt]
	\tikzstyle{n} = [draw=none, rectangle, inner sep=2pt] 
	\tikzstyle{c} = [draw, fill=black, circle, inner sep=0pt, minimum size=4pt]
    \centering
		\begin{tikzpicture}
		\node at (0,-1)[n]{$\m A_1$};
		\node(1) at (0,1)[c]{};
		\node(0) at (0,0)[c]{};
		\draw(0) edge (1);
		\end{tikzpicture}
		\begin{tikzpicture}
		\node at (0,-1)[n]{$\m A_2$};
		\node(2) at (0,2)[c]{};
		\node(1) at (0,1){};
		\node(0) at (0,0)[c]{};
		\draw(1) edge (2);
		\draw(0) edge (1);
		\end{tikzpicture}
        \
		\begin{tikzpicture}
		\node at (0,-1)[n]{$\m D_1$};
		\node(3) at (0,2)[c]{};
		\node(2) at (1,1){};
		\node(1) at (-1,1){};
		\node(0) at (0,0)[c]{};
		\draw(2) edge (3);
		\draw(1) edge (3);
		\draw(0) edge (1) edge (2);
		\end{tikzpicture}
        \ \ 
		\begin{tikzpicture}
		\node at (0,-1)[n]{$\m D_2$};
		\node(3) at (0,2)[c]{};
		\node(2) at (1,1)[c]{};
		\node(1) at (-1,1){};
		\node(0) at (0,0)[c]{};
		\draw(2) edge (3);
		\draw(1) edge (3);
		\draw(0) edge (1) edge (2);
		\end{tikzpicture}
		\ \ 
		\begin{tikzpicture}
		\node at (0,-1)[n]{$\m D_3$};
		\node(4) at (0,3)[c]{};
		\node(3) at (0,2){};
		\node(2) at (1,1){};
		\node(1) at (-1,1){};
		\node(0) at (0,0)[c]{};
		\draw(3) edge (4);
		\draw(2) edge (3);
		\draw(1) edge (3);
		\draw(0) edge (1) edge (2);
		\end{tikzpicture}
		\ \ 
		\begin{tikzpicture}
		\node at (0.5,-1)[n]{$\m D_4$};
		\node(5) at (1,3)[c]{};
		\node(4) at (2,2){};
		\node(3) at (0,2){};
		\node(2) at (1,1){};
		\node(1) at (-1,1){};
		\node(0) at (0,0)[c]{};
		\draw(4) edge (5);
		\draw(3) edge (5);
		\draw(2) edge (3) edge (4);
		\draw(1) edge (3);
		\draw(0) edge (1) edge (2);
		\end{tikzpicture}
		\begin{tikzpicture}
		\node at (0.5,-1)[n]{$\m D_5$};
		\node(5) at (1,3)[c]{};
		\node(4) at (2,2)[c]{};
		\node(3) at (0,2){};
		\node(2) at (1,1){};
		\node(1) at (-1,1){};
		\node(0) at (0,0)[c]{};
		\draw(4) edge (5);
		\draw(3) edge (5);
		\draw(2) edge (3) edge (4);
		\draw(1) edge (3);
		\draw(0) edge (1) edge (2);
		\end{tikzpicture}
		\begin{tikzpicture}
		\node at (0.5,-1)[n]{$\m D_6$};
		\node(5) at (1,3)[c]{};
		\node(4) at (2,2){};
		\node(3) at (0,2){};
		\node(2) at (1,1)[c]{};
		\node(1) at (-1,1){};
		\node(0) at (0,0)[c]{};
		\draw(4) edge (5);
		\draw(3) edge (5);
		\draw(2) edge (3) edge (4);
		\draw(1) edge (3);
		\draw(0) edge (1) edge (2);
		\end{tikzpicture}
		\begin{tikzpicture}
		\node at (0.5,-1)[n]{$\m D_7$};
		\node(5) at (1,3)[c]{};
		\node(4) at (2,2){};
		\node(3) at (0,2)[c]{};
		\node(2) at (1,1)[c]{};
		\node(1) at (-1,1){};
		\node(0) at (0,0)[c]{};
		\draw(4) edge (5);
		\draw(3) edge (5);
		\draw(2) edge (3) edge (4);
		\draw(1) edge (3);
		\draw(0) edge (1) edge (2);
		\end{tikzpicture}
		\
		\begin{tikzpicture}
		\node at (0,-1)[n]{$\m D_8$};
		\node(7) at (0,3)[c]{};
		\node(6) at (1,2){};
		\node(5) at (0,2){};
		\node(4) at (-1,2){};
		\node(3) at (1,1){};
		\node(2) at (0,1){};
		\node(1) at (-1,1){};
		\node(0) at (0,0)[c]{};
		\draw(6) edge (7);
		\draw(5) edge (7);
		\draw(4) edge (7);
		\draw(3) edge (5) edge (6);
		\draw(2) edge (4) edge (6);
		\draw(1) edge (4) edge (5);
		\draw(0) edge (1) edge (2) edge (3);
		\end{tikzpicture}
		\quad
		\begin{tikzpicture}
		\node at (0,-1)[n]{$\m D_9$};
		\node(7) at (0,3)[c]{};
		\node(6) at (1,2)[c]{};
		\node(5) at (0,2){};
		\node(4) at (-1,2){};
		\node(3) at (1,1){};
		\node(2) at (0,1){};
		\node(1) at (-1,1){};
		\node(0) at (0,0)[c]{};
		\draw(6) edge (7);
		\draw(5) edge (7);
		\draw(4) edge (7);
		\draw(3) edge (5) edge (6);
		\draw(2) edge (4) edge (6);
		\draw(1) edge (4) edge (5);
		\draw(0) edge (1) edge (2) edge (3);
		\end{tikzpicture}
		\quad
		\begin{tikzpicture}
		\node at (0,-1)[n]{$\m D_{10}$};
		\node(7) at (0,3)[c]{};
		\node(6) at (1,2){};
		\node(5) at (0,2){};
		\node(4) at (-1,2){};
		\node(3) at (1,1)[c]{};
		\node(2) at (0,1){};
		\node(1) at (-1,1){};
		\node(0) at (0,0)[c]{};
		\draw(6) edge (7);
		\draw(5) edge (7);
		\draw(4) edge (7);
		\draw(3) edge (5) edge (6);
		\draw(2) edge (4) edge (6);
		\draw(1) edge (4) edge (5);
		\draw(0) edge (1) edge (2) edge (3);
		\end{tikzpicture}
		\quad
		\begin{tikzpicture}
		\node at (0,-1)[n]{$\m D_{11}$};
		\node(7) at (0,3)[c]{};
		\node(6) at (1,2)[c]{};
		\node(5) at (0,2){};
		\node(4) at (-1,2){};
		\node(3) at (1,1)[c]{};
		\node(2) at (0,1){};
		\node(1) at (-1,1){};
		\node(0) at (0,0)[c]{};
		\draw(6) edge (7);
		\draw(5) edge (7);
		\draw(4) edge (7);
		\draw(3) edge (5) edge (6);
		\draw(2) edge (4) edge (6);
		\draw(1) edge (4) edge (5);
		\draw(0) edge (1) edge (2) edge (3);
		\end{tikzpicture}
		\quad
		\begin{tikzpicture}
		\node at (0,-1)[n]{$\m D_{12}$};
		\node(7) at (0,3)[c]{};
		\node(6) at (1,2)[c]{};
		\node(5) at (0,2){};
		\node(4) at (-1,2)[c]{};
		\node(3) at (1,1){};
		\node(2) at (0,1)[c]{};
		\node(1) at (-1,1){};
		\node(0) at (0,0)[c]{};
		\draw(6) edge (7);
		\draw(5) edge (7);
		\draw(4) edge (7);
		\draw(3) edge (5) edge (6);
		\draw(2) edge (4) edge (6);
		\draw(1) edge (4) edge (5);
		\draw(0) edge (1) edge (2) edge (3);
		\end{tikzpicture}
		\quad
		\begin{tikzpicture}
		\node at (0.5,-1)[n]{$\m D_{13}$};
		\node at (0,5)[n]{$ $};
		\node(7) at (2,4)[c]{};
		\node(6) at (3,3){}edge(7);
		\node(5) at (1,3){}edge(7);
		\node(4) at (2,2)[c]{}edge(5)edge(6);
		\node(3) at (0,2){}edge(5);
		\node(2) at (1,1){}edge(3)edge(4);
		\node(1) at (-1,1){}edge(3);
		\node(0)at(0,0)[c]{}edge(1)edge(2);
		\end{tikzpicture}
		\!\!\!\!\!
		\begin{tikzpicture}
		\node at (0.5,-1)[n]{$\m D_{14}$};
		\node at (0,5)[n]{$ $};
		\node(7) at (2,4)[c]{};
		\node(6) at (3,3){}edge(7);
		\node(5) at (1,3)[c]{}edge(7);
		\node(4) at (2,2)[c]{}edge(5)edge(6);
		\node(3) at (0,2){}edge(5);
		\node(2) at (1,1){}edge(3)edge(4);
		\node(1) at (-1,1){}edge(3);
		\node(0)at(0,0)[c]{}edge(1)edge(2);
		\end{tikzpicture}
		\quad
		\begin{tikzpicture}
		\node at (0.5,-1)[n]{$\m D_{15}$};
		\node at (0,5)[n]{$ $};
		\node(7) at (0,4)[c]{};
		\node(6) at (1,3){}edge(7);
		\node(5) at (-1,3){}edge(7);
		\node(4) at (2,2){}edge(6);
		\node(3) at (0,2)[c]{}edge(5)edge(6);
		\node(2) at (1,1)[c]{}edge(3)edge(4);
		\node(1) at (-1,1){}edge(3);
		\node(0)at(0,0)[c]{}edge(1)edge(2);
		\end{tikzpicture}
		\
		\begin{tikzpicture}
		\node at (0.5,-1)[n]{$\m D_{16}$};
		\node at (0,5)[n]{$ $};
		\node(7) at (0,4)[c]{};
		\node(6) at (1,3)[c]{}edge(7);
		\node(5) at (-1,3){}edge(7);
		\node(4) at (2,2){}edge(6);
		\node(3) at (0,2)[c]{}edge(5)edge(6);
		\node(2) at (1,1)[c]{}edge(3)edge(4);
		\node(1) at (-1,1){}edge(3);
		\node(0)at(0,0)[c]{}edge(1)edge(2);
		\end{tikzpicture}
    \caption{Subdirectly irreducible $d\ell \p$-closure algebras up to 8 elements (black = closed).}
    \label{fig:sidlp}
\end{figure}
In \cite{P1996} and \cite{P1999} a characterization of the simple and subdirectly irreducible $d\ell \p$-chains, or totally ordered modal lattices, is given. Recall that a $d\ell \p$-chain is an algebra $(L, \wedge, \vee, \bot, \top,\p)$ that is a linearly ordered bounded distributive lattice with normal unary operator $\p $. We denote the following $d\ell \p$-chains by $\mathbf{A}_k$, $\mathbf{B}_k$, and $\mathbf{B'}_k$ for any integer $k \geq  1$:

Let $\top=a_0>a_1>a_2>\dots>\bot$ and $\bot=b_0<b_1<b_2<\dots<\top$ be bounded countable decreasing and increasing chains respectively.
Then the operator $\p $ is defined in each structure as follows:

In $\mathbf{A}_k=\{\top,a_1,\dots,a_{k-1},\bot\}$, $\p a_i = a_{i-1}$ for $1 \leq i < k$, $\p\top = \top$, and $\p\bot = \bot$.

In $\mathbf{B}_k=\{\bot,b_1,\dots,b_{k-1},\top\}$, $\p b_i = b_{i-1}$ for $1 \leq i < k$, $\p\top = \top$, and $\p\bot = \bot$.

In $\mathbf{B}'_k=\{\bot,b_1,\dots,b_{k-1},\top\}$, $\p b_i = b_{i-1}$ for $1 \leq i < k$, $\p\top=b_{k-1}$ and $\p\bot = \bot$.

In $\mathbf{A}_\infty=\{\top,a_1,a_2,\dots,\bot\}$, $\p a_i = a_{i-1}$ for $1 \leq i$, $\p\top = \top$, and $\p\bot = \bot$.

In $\mathbf{B}_\infty=\{\bot,b_1,b_2,\dots,\top\}$, $\p b_i = b_{i-1}$ for $1 \leq i$, $\p\top = \top$, and $\p\bot = \bot$.

\begin{thmC}[\cite{P1996}] \hfill
\begin{enumerate}
    \item The simple $d\ell \p$-chains are the algebras $\mathbf{A}_1$ and $\mathbf{B'}_1$.
    \item The subdirectly irreducible $d\ell \p$-chains are the algebras $\mathbf{A}_k, \mathbf{B}_k, \mathbf{B'}_k$ for every natural number $k$ and $\mathbf{A}_\infty, \mathbf{B}_\infty$.
\end{enumerate}
\end{thmC}
These chains are pictured in Figure \ref{fig:sidlpch}. All subdirectly irreducible $d\ell \p$-closure algebras up to cardinality $8$ are shown in Figure~\ref{fig:sidlp}. Note that $\m D_{12}$ does not satisfy the identity $\p x\wedge \p y\le \p((\p x\wedge y)\vee(x\vee \p y))$, hence by Lemma~\ref{five} the corresponding unary-determined magma is not associative.

\begin{cor} Linear $d\ell \p$-closure algebras of size $n \geq 4$ are not subdirectly irreducible.
\end{cor}
\begin{proof}
Let $\mathbf{W}$ be a linearly-ordered preorder-forest $P$-frame with corresponding linear $d\ell \p$-closure algebra $D(\mathbf{W})$. Suppose that $D(\mathbf{W})$ is subdirectly irreducible. Then $D(\mathbf{W})$ is of the form $\mathbf{A}_k$, $\mathbf{B}_k$, or $\mathbf{B'}_k$ for some $k$ in the natural numbers, or $D(\mathbf{W})$ is of the form $\mathbf{A}_\infty$ or $\mathbf{B}_\infty$. 

By Lemma~\ref{frameprops}, since $P$ is reflexive, $X \leq \p X$ for all $X \in D(\mathbf{W})$. But in $\mathbf{B}_\infty$ or $\mathbf{B}_k$ with $k \geq 2$, there exists $X$ such that $\p X < X$, so $D(\mathbf{W})$ cannot be of this form.

We also have that $\p\top = \top$ in all $d\ell \p$-closure algebras, so we cannot have $D(\mathbf{W}) = C_k$ for any $k > 0$.

Now suppose that $D(\mathbf{W}) = \mathbf{A}_k$ where $k$ is a natural number or $\infty$. Suppose $k \geq 3$. Then there exist $X_i, X_{i+1}, X_{i+2} \in D(\mathbf{W})$ such that $\p X_i = X_{i+1} \neq X_{i+2} = \p \p X_i$. Hence $\p $ is not a closure operator, a contradiction.
\end{proof}

Hence the only subdirectly irreducible linear $d\ell \p$-closure algebras are $\mathbf{A}_1$, and $\mathbf{A}_2$, pictured in Figure \ref{fig:sidlp}. Since $d\ell \p$-algebras have lattice reducts, the variety of all $d\ell \p$-algebras is congruence distributive, and it follows from J\'onsson's Lemma \cite{Jon1967} that nonisomorphic finite subdirectly irreducible $d\ell \p$-algebras generate distinct varieties. Moreover, these varieties are completely join-irreducible elements of the lattice of all varieties. A diagram of the poset of join-irreducible varieties generated by $d\ell \p$-chains and the algebras $\m D_1$--$\m D_{16}$ is shown in Figure~\ref{fig:jidlp}. The variety generated by an algebra $\m A$ is denoted by $\mathcal A=\mathbb V(\m A)$. Equational bases for the varieties generated by bounded $d\ell \p$-chains are given in \cite{P1996}.
\begin{figure}
	\tikzstyle{every picture} = [scale=1]
	\tikzstyle{every node} = [draw, fill=white, ellipse, inner sep=0pt, minimum size=4pt]
	\tikzstyle{n} = [draw=none, rectangle, inner sep=2pt] 
    \centering
		\begin{tikzpicture}
		\node(D16) at (-1,4)[n]{$\mathcal D_{16}$};
		\node(D14) at (-1.75,4)[n]{$\mathcal D_{14}$};
		\node(D15) at (-4,4)[n]{$\mathcal D_{15}$};
		\node(D13) at (-4.75,4)[n]{$\mathcal D_{13}$};
		\node(D10) at (-7,5)[n]{$\mathcal D_{10}$};
		\node(D8) at (-8,5)[n]{$\mathcal D_{8}$};
		\node(D4) at (-7,4)[n]{$\mathcal D_4$}edge(D8)edge(D10);
		\node(D11) at (-2.5,4)[n]{$\mathcal D_{11}$};
		\node(D12) at (-1,3)[n]{$\mathcal D_{12}$};
		\node(D9) at (-3.25,4)[n]{$\mathcal D_{9}$};
		\node(D7) at (-2,3)[n]{$\mathcal D_{7}$}edge(D11)edge(D14)edge(D16);
		\node(D5) at (-3,3)[n]{$\mathcal D_{5}$}edge(D9)edge(D11);
		\node(D6) at (-4,3)[n]{$\mathcal D_{6}$}edge(D10)edge(D13)edge(D15);
		\node(D3) at (-5.5,3)[n]{$\mathcal D_3$}edge(D4);
		\node(D2) at (-2.5,2)[n]{$\mathcal D_2$}edge(D6)edge(D5)edge(D7)edge(D12);
		\node(D1) at (-4,2)[n]{$\mathcal D_1$}edge(D3)edge(D6);
		\node(C5) at (2,4)[n]{$\vdots$};
		\node(C4) at (2,3)[n]{$\mathcal B'_4$};
		\node(C3) at (1.5,2)[n]{$\mathcal B'_3$}edge(C4);
		\node(C2) at (1,1)[n]{$\mathcal B'_2$}edge(C3);
		\node(C1) at (0.5,0)[n]{$\mathcal B'_1$}edge(C2);
		\node(B5) at (1,4)[n]{$\vdots$};
		\node(B4) at (1,3)[n]{$\mathcal B_4$};
		\node(B3) at (.5,2)[n]{$\mathcal B_3$}edge(B4);
		\node(B2) at (0,1)[n]{$\mathcal B_2$}edge(B3);
		\node(A5) at (0,4)[n]{$\vdots$};
		\node(A4) at (0,3)[n]{$\mathcal A_4$};
		\node(A3) at (-.5,2)[n]{$\mathcal A_3$}edge(A4);
		\node(A2) at (-1,1)[n]{$\mathcal A_2$}edge(D1)edge(D2)edge(A3);
		\node(A1) at (-.5,0)[n]{$\mathcal{DL}$}edge(A2)edge(B2);
		\draw[ultra thick](D1)--(D3);
		\draw[ultra thick](D2)--(D6);
		\draw[ultra thick](D2)--(D7);
		\draw[ultra thick](C1)--(C2)--(C3)--(C4);
		\end{tikzpicture}
    \caption{Some join-irreducible varieties of $d\ell \p$-closure algebras and bounded $d\ell \p$-chains ordered by inclusion. Lines are thin if $\m A\in\mathbb{S}(\m B)$ and thick if $\m A\in\mathbb{HS}(\m B)$ for generating algebras $\m A,\m B$.}
    \label{fig:jidlp}
\end{figure}
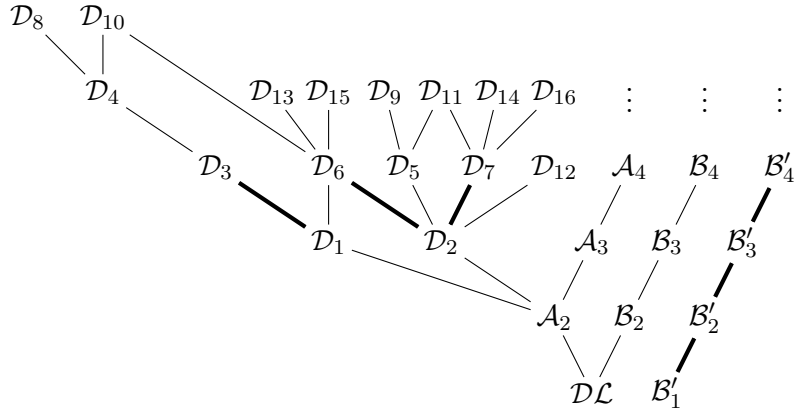

Varieties of unary-determined bunched implication algebras are obtained from Heyting algebras with a residuated closure operator (Corollary~\ref{BItermeq}). For a Heyting algebra, the congruence lattice is isomorphic to the set of filters (ordered by reverse inclusion). Hence the subdirectly irreducible Heyting algebras are characterized by having a unique coatom. In particular, all finite Heyting chains are subdirectly irreducible which leads to the following result.

\begin{thm}
All finite Heyting chains with additional operations are subdirectly irreducible. This includes all finite bunched implication chains and all finite Heyting chains with residuated closure operators.
\end{thm}

According to Theorem~\ref{chains} there are $2^{n-2}$ unary-determined commutative doubly idempotent linear semirings with $n$ elements, and if they are expanded with a Heyting implication (i.e. a residual of the meet operation) they are term-equivalent to $2^{n-2}$ Heyting chains with a residuated closure operator. Bunched implication algebras have an identity element, so in this variety there are $n-1$ subdirectly irreducible unary-determined bunched implication (BI) chains with $n$ elements, denoted by $\m C_{nk}$ for $1\le k<n$. The structure of these chains is described in the proof of Theorem~\ref{chains} and illustrated on the left in Figure~\ref{fig:bichains}. 

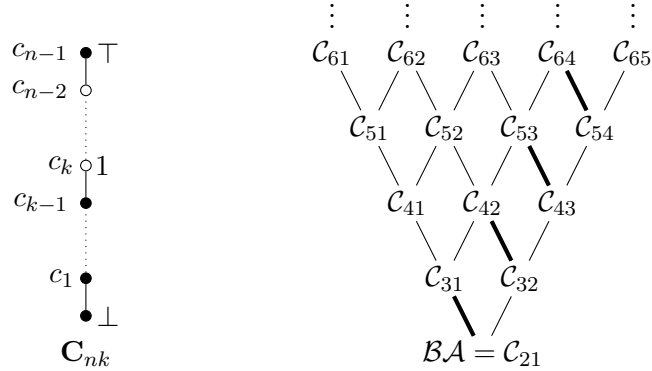
\begin{figure}
	\tikzstyle{every node} = [draw, fill=white, ellipse, inner sep=0pt, minimum size=4pt]
	\tikzstyle{n} = [draw=none, rectangle, inner sep=2pt] 
	\tikzstyle{c} = [draw, fill=black, circle, inner sep=0pt, minimum size=4pt]
        \centering
		\begin{tikzpicture}[scale=.5]
		\node at (0,-1)[n]{$\m C_{nk}$};
		\node(5)at(0,7)[c,label=left:$c_{n-1}$,label=right:$\top$]{};
		\node(4)at(0,6)[label=left:$c_{n-2}$]{}edge(5);
		\node(3)at(0,4)[label=left:$c_{k}$,label=right:$1$]{}edge[dotted](4);
		\node(2)at(0,3)[c,label=left:$c_{k-1}$]{}edge(3);
		\node(1)at(0,1)[c,label=left:$c_{1}$]{}edge[dotted](2);
		\node(0)at(0,0)[c,label=right:$\bot$]{}edge(1);
		\end{tikzpicture}
		\qquad\qquad\qquad
		\begin{tikzpicture}[scale=1]
		\node(C5) at (2,3.6)[n]{$\vdots$};
		\node(C4) at (1,3.6)[n]{$\vdots$};
		\node(C3) at (0,3.6)[n]{$\vdots$};
		\node(C2) at (-1,3.6)[n]{$\vdots$};
		\node(C1) at (-2,3.6)[n]{$\vdots$};
		\node(C65) at (2,3)[n]{$\mathcal C_{65}$};
		\node(C64) at (1,3)[n]{$\mathcal C_{64}$};
		\node(C63) at (0,3)[n]{$\mathcal C_{63}$};
		\node(C62) at (-1,3)[n]{$\mathcal C_{62}$};
		\node(C61) at (-2,3)[n]{$\mathcal C_{61}$};
		\node(C54) at (1.5,2)[n]{$\mathcal C_{54}$}edge(C64)edge(C65);
		\node(C53) at (.5,2)[n]{$\mathcal C_{53}$}edge(C63)edge(C64);
		\node(C52) at (-.5,2)[n]{$\mathcal C_{52}$}edge(C62)edge(C63);
		\node(C51) at (-1.5,2)[n]{$\mathcal C_{51}$}edge(C61)edge(C62);
		\node(C43) at (1,1)[n]{$\mathcal C_{43}$}edge(C53)edge(C54);
		\node(C42) at (0,1)[n]{$\mathcal C_{42}$}edge(C52)edge(C53);
		\node(C41) at (-1,1)[n]{$\mathcal C_{41}$}edge(C51)edge(C52);
		\node(C32) at (0.5,0)[n]{$\mathcal C_{32}$}edge(C42)edge(C43);
		\node(C31) at (-.5,0)[n]{$\mathcal{C}_{31}$}edge(C41)edge(C42);
		\node(C21) at(0,-1)[n]{$\mathcal{BA}=\mathcal C_{21}$}edge(C31)edge(C32);
		\draw[ultra thick](C31)--(C21);
		\draw[ultra thick](C42)--(C32);
		\draw[ultra thick](C53)--(C43);
		\draw[ultra thick](C64)--(C54);
		\end{tikzpicture}
    \caption{All finite subdirectly irreducible unary-determined BI-chains (black elements are closed) and the poset of join-irreducible varieties they generate.}
    \label{fig:bichains}
\end{figure}

The variety generated by linearly ordered Heyting algebras is also known as the variety of G\"odel algebras, and it has a countable chain of subvarieties, each generated by a finite G\"odel chain. The BI-chains $\m C_{n,n-1}$ generate this chain of subvarieties since they satisfy $\p x=x$, i.e. all their elements are closed and $1=\top$.

From the structure of the subdirectly irreducible BI-chains $\m C_{nk}$ one can observe the following result.

\begin{thm}
For $n>1$ and $k\ge 1$, each BI-chain $\m C_{n,k}$ is embedded in $\m C_{n+1,k+1}$.

\noindent
For $n>2$ and $k\ge 1$, each BI-chain $\m C_{n,k}$ is embedded in $\m C_{n+1,k}$.

\noindent
For $n>2$, each BI-chain $\m C_{n,n-2}$ maps homomorphically onto $\m C_{n-1,n-2}$.
\end{thm}

Based on this result, the poset of join-irreducible varieties of bunched implication algebras that are generated by finite unary-determined BI-chains is shown on the right in Figure~\ref{fig:bichains}. Note that the two-element BI-chain $\m C_{2,1}$ is term-equivalent to the two-element Boolean algebra and generates the smallest nontrivial variety.

\section{Conclusion}

We showed that unary-determined $d\ell$-magmas have a simple algebraic
structure given by two unary operators and that their relational frames are definitionally equivalent to frames with two binary relations. The complex algebras of these frames are complete distributive lattices with completely distributive operators, hence they have residuals and can be considered Kripke semantics for unary-determined bunched implication algebras and bunched implication logic.
Associativity of the binary operator for idempotent unary-determined algebras can be checked by an identity with 2 rather than 3 variables, and for the frames by a 3-variable universal formula rather than a 6-variable universal-existential formula. All idempotent Boolean magmas are unary-determined, hence these results significantly extend the structural characterization of idempotent atomic Boolean quantales in \cite{AJ2020} and relate them to bunched implication logic. As an application we counted the number of preorder forest $P$-frames with $n$ elements for which the partial order is an antichain, as well as the number of linearly ordered preorder $P$-frames. We also found all subdirectly irreducible $d\ell \p$-closure algebras up to cardinality 8, as well as all finite subdirectly irreducible unary-determined BI-chains and showed how the varieties they generate are related to each other by subclass inclusion. 

\textbf{Acknowledgements.}
The investigations in this paper made use of Prover9/Mace4 \cite{McC}. In particular,
parts of Lemma~\ref{pqprops} and Theorem~\ref{equivframes} were developed with the help of Prover9 (short proofs were extracted from the output) and the results in Table~\ref{nofmagmas} were calculated with Mace4. The remaining results in Sections 2--4 were proved manually, and later also checked with Prover9.

\bibliographystyle{alphaurl}
\bibliography{refs}
\end{document}